\documentclass[10pt]{article}
\usepackage[utf8]{inputenc}
\usepackage[a4paper, left=2cm,
  right=1.5cm,
  top=2.5cm,
  bottom=2.5cm]{geometry}

\usepackage{amsmath} 
\usepackage{bbm}
\usepackage{amsfonts,amssymb,amsthm}
\usepackage{graphicx}
\usepackage[colorlinks,linkcolor=Blue,
            anchorcolor=red,
            citecolor=orange,]{hyperref}
\usepackage{mathrsfs}
\usepackage{enumerate}
\usepackage{indentfirst}
\usepackage{geometry}
\usepackage{url}
\usepackage[usenames,dvipsnames,svgnames,table]{xcolor}

\allowdisplaybreaks[4]
\numberwithin{equation}{section}

\newcommand{\dd}{\mathrm{d}}
\newcommand{\ind}{\mathbbm{1}}

\newtheorem{definition}{Definition}[section]
\newtheorem{theorem}[definition]{\textbf{Theorem}}
\newtheorem{lemma}[definition]{\textbf{Lemma}}
\newtheorem{proposition}[definition]{\textbf{Proposition}}
\newtheorem{corollary}[definition]{\textbf{Corollary}}
\newtheorem{remark}[definition]{\textbf{Remark}}

\newcommand{\ee}{\mathrm{e}} 
\newcommand{\bN}{\mathbb{N}}

\newcommand{\bR}{\mathbb{R}}

\newcommand{\bE}{\mathbb{E}} %
\newcommand{\bP}{\mathbb{P}} 

\newcommand{\cM}{\mathcal{M}} 

\newcommand{\rF}{\mathscr{F}} 
\newcommand{\rG}{\mathscr{G}} 

\newcommand{\fN}{\mathbf{N}} 
\newcommand{\fT}{\mathbf{T}}

\renewcommand{\b}{\mathfrak b} 
\newcommand{\m}{\mathfrak m}

\title{Perturbed Brownian motion reflected at a time-dependent boundary}
\author{Chengshi Wang\footnote{SMS, Fudan University, \texttt{cswang17@fudan.edu.cn}}}
\date{}

\begin{document}

\maketitle

\begin{abstract}
Let $B$ be a standard Brownian motion, $x\ge 0,\, \nu<1$ and $b:[0,\infty)\to\bR$ be a continuous function locally of finite variation starting from 0.  We define the perturbed Brownian motion reflected at the boundary $b$ by establishing strong existence and pathwise uniqueness of a solution to the equation
\[
 W_t=(1-\nu)x+B_t+\nu M_t(W)+\frac12 L_t^0(W-b),
 \qquad W_t\ge b(t),
\]
where $M_t(W):=\sup_{0\le s\le t} W_s$ and the process $L^0(W-b)$ is the semimartingale local time at 0 of the process $W-b$.  We give a positive result under condition (PB) on the boundary $b$: 
for every $T>0$, the upward increment $\sup_{0\le s<t\le T,\,t-s\le h}(b(t)-b(s))^+ = o(\sqrt{h})$ as $h\downarrow 0$.
The proof splits into two regimes: the case $\nu<1/2$ is a consequence of the Skorokhod problem in an orthant proved by Williams \cite{Williams1995}, while the case $\nu\ge 1/2$ combines a deterministic comparison estimate and a logarithmic upper bound on the number of completed round-trips, following the strategy of Chaumont and Doney \cite{ChaumontDoney}. 
For $\alpha\in(0,1/2)$, we also construct an increasing $\alpha$-H\"older boundary for which no continuous adapted solution starting from zero exists for any $\nu<1$.

\bigskip


\medskip

\end{abstract}

\section{Introduction}

Let $(\Omega,\rF,(\rF_t)_{t\ge0},\bP)$ satisfy the usual conditions and let $B$ be an $\rF$-Brownian motion. Let $x\ge 0,\, \nu<1$ and $b:[0,\infty)\to \bR$ be a continuous function locally of finite variation starting from 0. We study the strong existence and pathwise uniqueness of a solution to the following equation
\begin{equation}\label{eq:main}
 W_t=(1-\nu)x+B_t+\nu M_t(W)+\frac12 L_t^0(W-b),
 \qquad W_t-b(t)\ge0,\qquad t\ge 0,
\end{equation}
where $ M_t(w):=\sup_{0\le s\le t}w_s$ for a continuous function $w:[0,\infty)\to \bR$, and $L_t^a(X)$ is the local time of a continuous $\rF$-adapted process $X=(X_t:t\ge 0)$ at time $t$ and level $a$, normalized by
\[
 L_t^{a}(X) =\lim_{\varepsilon\downarrow0}\frac1\varepsilon \int_0^t\ind_{[a,a+\varepsilon)}(X_s)\dd\langle X\rangle_s,\qquad t\ge 0.
\]
The solution $W$ is called a perturbed Brownian motion reflected at the time-dependent boundary $b$; it evolves as a Brownian motion away from its running maximum and the boundary $b$, while it is perturbed when attaining a new maximum and reflected at the boundary. 
When $\nu=0$ and $b\equiv0$, \eqref{eq:main} reduces to the standard equation for reflecting Brownian motion started at $x$. For $x=0$, Tanaka's formula and Skorokhod's lemma also yield L\'evy's identity; see \cite{RevuzYor}.

The case $b\equiv 0$ has attracted much attention because of a deep connection with the perturbed reflecting Brownian motion (also called $\mu$-processes), which has been the subject of numerous work including \cite{LeGall-Yor1986,Yor1992,CarmonaPetitYor1994,Werner1995,Perman1996,PermanWerner,AidekonHuShi2021,AidekonHuShi2024,AidekonHuShi2026}.
In the zero-initial-state case, Le Gall and Yor \cite{LeGall-Yor1990} established the uniqueness of an adapted solution of \eqref{eq:main} for $\nu<1/2$. The range $1/2\le \nu<1$ was subsequently treated by Davis \cite{Davis1999} and Chaumont and Doney \cite{ChaumontDoney}, using different methods.
This model is also related to Brownian motions perturbed at extrema, which have been studied in \cite{Davis1996,PermanWerner,Davis1999,ChaumontDoney,ChaumontDoneyCalculations,ChaumontDoneyHu}; these processes arise as scaling limits of certain self-interacting random walks, see for example \cite{DolgopyatKosygina2012,KosyginaMountfordPeterson2022,KosyginaMountfordPeterson2023,LiuWang}. Doney and Zhang \cite{DoneyZhang} studied the deterministic perturbed Skorokhod problem with reflection at 0 and established corresponding results for perturbed diffusion processes.
Bahaj and Hiderah \cite{BahajHiderah2024} proved the existence and uniqueness of a strong solution of a perturbed SDE with a local time term under suitable assumptions on the coefficients. 

For a general time-dependent boundary $b$, the reflection in
\eqref{eq:main} is related to Skorokhod problems in evolving domains. We refer to \cite{BurdzyKangRamanan} for existence and uniqueness for the extended Skorokhod problem
between two time-dependent boundaries and \cite{NystromOnskog} for Skorokhod problems with oblique reflection in time-dependent domains.
On the other hand, set $X=W-b,\, Y=M(W)-W$ and $K= L^0(W-b)/2,\, V=(1-\nu)(M(W)-x)$. Then \eqref{eq:main} gives
\begin{align}\label{eq:2D-Skorokhod}
 \begin{pmatrix}X_t\\Y_t\end{pmatrix}
 =\begin{pmatrix}x+B_t-b(t)\\-B_t\end{pmatrix}
 +R\begin{pmatrix}K_t\\V_t\end{pmatrix},
 \qquad \text{where}~
 R=\begin{pmatrix}1&\nu^*\\-1&1\end{pmatrix},\ \nu^*:=\frac{\nu}{1-\nu}>-1.
\end{align}
The measure $\dd K$ is a.s. carried by $\{X=0\}$ and $\dd V$ by $\{Y=0\}$. Equation \eqref{eq:2D-Skorokhod} is a specific example of the Skorokhod problem in an orthant: given a continuous function $f:[0,\infty)\to \bR^d$ with $f(0)\in\bR_+^d$ and a $d\times d$ real-valued matrix $R$, the problem is to find continuous functions $g,m:[0,\infty)\to\bR_+^d$, such that $g=f+Rm$, $m$ is increasing and starts from 0 and each component $\dd m_i$ is carried by $\{g_i=0\}$ for each $i=1,\cdots,d$. Classical references for Skorokhod problems and reflected
Brownian motion in an orthant include \cite{HarrisonReiman,BernardElKharroubi,TaylorWilliams,Williams1995,DaiWilliams}.
Write $Q=I-R$ for $I$ a $d$-dimensional unit matrix, and let $|Q|$ be obtained by replacing each entry by its absolute value. The existence and uniqueness of this problem hold whenever $\rho(|Q|)$, the spectral radius of $|Q|$, is strictly smaller than 1, as shown in \cite{Williams1995}. More generally, the existence of a solution for every continuous driver $f$ is equivalent to the so-called completely-$\mathcal S$ condition on the matrix $R$, see \cite{BernardElKharroubi,DaiWilliams}; the uniqueness depends further on the driver $f$ under this condition.
When $R$ is completely-$\mathcal S$ with each diagonal entry equal to 1 and $f$ is a standard $d$-dimensional Brownian motion, Bass and Burdzy \cite{BassBurdzyCritical} proved pathwise uniqueness and the existence of an adapted solution in the critical case, i.e. $\rho(|Q|)=1$; their companion work \cite{BassBurdzyNonunique} exhibits non-uniqueness for almost every Brownian driving path in certain supercritical two-dimensional configurations, where $\rho(|Q|)>1$. Note that in dimension $d=2$, partial results for the deterministic problem in critical and supercritical cases have already been obtained in an unpublished paper by Mandelbaum; see \cite{Bass2024}.

\bigskip

We now state the main result of this paper. 
To be brief, we extend the fixed-boundary model to deterministic time-dependent boundaries and identify a one-sided regularity condition ensuring strong existence and pathwise uniqueness of equation \eqref{eq:main}.
For $T<\infty$ and $h>0$, put
\begin{equation}\label{eq:positive-modulus}
 \omega^+_{b,T}(h):=
 \sup_{\substack{0\le s<t\le T\\t-s\le h}}(b(t)-b(s))^+.
\end{equation}

\begin{theorem}\label{thm:main}
Let $x\ge0,\, \nu<1$ and $b:[0,\infty)\to\bR$ be continuous and locally of finite variation, with $b(0)=0$, and suppose that for every $T>0$,
\begin{equation}\label{eq:PB}
 \tag{PB}
 \frac{\omega^+_{b,T}(h)}{\sqrt h}\to 0 \qquad \text{as}~ h\downarrow 0.
\end{equation}
Then \eqref{eq:main} has a pathwise unique strong solution.  
\end{theorem}
 
Condition \eqref{eq:PB} controls only upward increments: $\omega^+_{b,T}(h)=o(\sqrt h)$ as $h\downarrow 0$. It imposes no additional restriction on downward increments. Within the class of continuous boundaries of locally finite variation, it holds whenever $b$ is locally H\"older continuous of order $\alpha>1/2$, and it also holds for every decreasing boundary.
An adapted solution of \eqref{eq:main} starting from a positive state $x>0$ can be constructed uniquely following the method given by \cite{LeGall-Yor1990,ChaumontDoney} until time $\zeta:=\inf\{t>0:M_t(W)=b(t)\}$; Proposition~\ref{prop:corner-power} then shows that $\zeta=\infty$ under \eqref{eq:PB}, so the solution exists for all $t\ge 0$. 
While the case $\nu<1/2$ can be dealt with using \eqref{eq:2D-Skorokhod}, the strategy of the proof of Theorem~\ref{thm:main} in the case $x=0,\nu\in[1/2,1)$ follows \cite{ChaumontDoney}: we construct a Cauchy sequence of solutions with drivers that coincide with $B$ after an $\rF$-stopping time tending to 0, and show that the local uniform limit solves \eqref{eq:main}.
The assumption that $b$ is locally of finite variation together with \eqref{eq:PB} then identifies the regulator $K$ with half of the semimartingale local time of $W-b$.

The next result shows that local finite variation and H\"older continuity of order below $1/2$ do not suffice for existence of a solution of \eqref{eq:main}.

\begin{theorem}\label{thm:nonexistence}
For every $\alpha\in(0,1/2)$, there is a deterministic continuous
increasing function $b:[0,\infty)\to[0,1]$, with $b(0)=0$ and
total variation $1$, such that:
\begin{enumerate}[(i)]
\item $b$ is globally $\alpha$-H\"older continuous and its Stieltjes
measure is singular with respect to Lebesgue measure;
\item for every $\nu<1$ and $T>0$, \eqref{eq:main} with
$x=0$ has no continuous adapted solution on $[0,T]$.
\end{enumerate}
\end{theorem}

Note that when $\nu<1/2$, the spectral radius of $|Q|$ is $\sqrt{|\nu^*|}<1$, so the Skorokhod problem in an orthant has a unique solution. We point out that there is no contradiction with Theorem~\ref{thm:nonexistence}: the function $K$ satisfying \eqref{eq:2D-Skorokhod} need \emph{not} be half of the local time of $W-b$. As shown in Section~\ref{sec:example}, with the boundary $b$ increasing, for any solution $W$ of \eqref{eq:main}, the contact set with the boundary $\{t\ge 0:W_t=b(t)\}$ has null $\dd b$-measure; the solution of \eqref{eq:2D-Skorokhod} with our construction on $b$ does not satisfy this necessary condition.

\bigskip

The organization of the paper is as follows. In Section~\ref{sec:positive}, we construct solutions starting from positive initial states and prove the strict separation of the running maximum from the boundary under \eqref{eq:PB}. Section~\ref{sec:comparison} establishes a deterministic comparison estimate in terms of the number of round-trips, while Section~\ref{sec:cycles} proves a logarithmic upper bound on this number for solutions started near the boundary. We combine these estimates to prove Theorem~\ref{thm:main} in Section~\ref{sec:proof} and construct the counterexample of Theorem~\ref{thm:nonexistence} in Section~\ref{sec:example}.

\paragraph{Acknowledgment}
We thank Elie A\"id\'ekon, Lo\"ic Chaumont and Dongjian Qian for helpful discussions.

\paragraph{Declaration of AI use}
The author used a generative AI tool to improve the language, identify typographical errors and assist in developing certain proofs, especially the key estimate concerning the separation of maxima from the boundary (Proposition~\ref{prop:corner-power}) and the construction of a singular increasing boundary for the counterexample in Theorem~\ref{thm:nonexistence}. The author takes full responsibility for the content of the paper.


\section{Solution from a positive initial state}\label{sec:positive}

Let $e,b:[0,\infty)\to\bR$ be continuous with $e(0)=b(0)=0$. For $m\ge x\ge 0$ and $\nu<1$, consider $(w,k^w)$ satisfying
\begin{equation}\label{eq:deterministic-regulator}
\begin{aligned}
 w_t &=x+e_t+\nu \big((M_t(w)\vee m) - m\big) +k_t^w,\qquad  w\ge b \qquad t\ge 0,
\\
 k_0^w &=0,\quad k^w\text{ continuous and increasing},
 \qquad \int\ind_{\{w_t>b(t)\}}\dd k_t^w=0.
\end{aligned}
\end{equation}
This is a deterministic version of \eqref{eq:main}. Set
\begin{equation}\label{eq:nu}
 \bar{\nu}:=1-\nu>0,
 \qquad \nu^*:=\frac{\nu}{\bar{\nu}}\in(-1,\infty).
\end{equation}
We first prove the following simple identities, showing that a solution of \eqref{eq:deterministic-regulator} defined on some finite interval $[0,\zeta)$ can be continuously extended to time $\zeta$. Set $x^+ = \max\{x,0\}$, $x^-=\max\{-x,0\}$.

\begin{lemma}\label{lem:identities}
Let $(w_t,k^w_t:0\le t<\zeta)$ solve \eqref{eq:deterministic-regulator} with $0\le x\le m$ and some $\zeta\in [0,\infty]$. For all $t\in [0,\zeta)$,
\begin{align}
 & \bar{\nu}\bigl((M_t(w) \vee m)-m\bigr) =\left[\sup_{0\le s\le t}\{e_s+k^w_s\}-(m-x)\right]^+, \label{eq:max-identity}\\
 & k^w_t=\sup_{0\le s\le t} \bigl[b(s)-x-e_s-\nu((M_s(w) \vee m)-m)\bigr]^+.
 \label{eq:K-identity}
\end{align}
In particular,
\begin{equation}\label{eq:K-rough}
 0\le k^w_t\le \left[\sup_{0\le s\le t}b(s)-x-\inf_{0\le s\le t}e_s\right]^+ + \nu^-\bigl((M_t(w) \vee m)-m\bigr).
\end{equation}
If $\zeta<\infty$, the solution $(w,k^w)$ can be extended to time $\zeta$.
\end{lemma}

\begin{proof}
We have \(e_t+k^w_t\le \bar{\nu}((M_t(w) \vee m)-m)+(m-x)\) from \eqref{eq:deterministic-regulator}. If $M_t(w)\le m$, this upper bound proves \eqref{eq:max-identity}. Otherwise, the continuity of $w$ gives a time $s_t\le t$ at which $w_{s_t}=M_{s_t}(w)=M_t(w)>m$. Substituting it into
\eqref{eq:deterministic-regulator} yields $e_{s_t}+k^w_{s_t}=\bar{\nu}(M_t(w)-m)+(m-x)$, which together with the preceding upper bound proves \eqref{eq:max-identity}.
Subtracting $b(t)$ from both sides of \eqref{eq:deterministic-regulator}, we obtain \eqref{eq:K-identity} by Skorokhod's lemma \cite[Lemma VI.2.1]{RevuzYor}. Therefore,
\[
 \begin{aligned}
 k^w_t&\le\sup_{s\le t}[b(s)-x-e_s]^+ + \nu^-\bigl((M_t(w) \vee m)-m\bigr)\\
 &\le\left[\sup_{s\le t}b(s)-x-\inf_{s\le t}e_s\right]^+ + \nu^-\bigl((M_t(w) \vee m)-m\bigr),
 \end{aligned}
\]
which is \eqref{eq:K-rough}. Let us deal with the case $\zeta<\infty$. If $\nu \ge 0$, \eqref{eq:K-rough} gives an upper bound for $k^w$ on $[0,\zeta)$, and thus $M(w)$ is bounded from above by applying \eqref{eq:max-identity}. If $\nu<0$, we deduce from \eqref{eq:max-identity}--\eqref{eq:K-rough} that
\[
(M_t(w) \vee m) \le m + \sup_{s\le t}e(s) + \left[\sup_{s\le t}b(s)-x-\inf_{s\le t}e_s\right]^+, \qquad t\in[0,\zeta).
\]
Therefore $M(w)$ is bounded from above, and so is $k^w$ by again applying \eqref{eq:K-rough}. Since $M(w)$ and $k^w$ are continuous and increasing, these two processes can be continuously extended to time $\zeta$, which, from \eqref{eq:deterministic-regulator}, yields an extension for $w$ and completes the proof of the lemma.
\end{proof}

\begin{corollary}\label{cor:max}
Let $(w_t,k^w_t:0\le t<\zeta)$ solve \eqref{eq:deterministic-regulator} with $x=m\ge 0$ and some $\zeta\in [0,\infty]$.
Then if $e_t=M_t(e)$ for some $t\le \zeta$ finite, then $M_t(w) = w_t$.
In particular, writing $T_a(e)=\inf\{t>0:e_t=a\}$ for $a>0$, $w_{T_a(e)}=M_{T_a(e)}(w)$ for every $T_a(e)\le \zeta$ finite.  
\end{corollary}

\begin{proof}
Note that when $\zeta<\infty$, the solution can be continuously extended to $\zeta$ by Lemma~\ref{lem:identities}. For $x=m$, \eqref{eq:max-identity} reduces to
\begin{align}\label{eq:max-identity-simplified}
    \bar{\nu}\bigl(M_t(w) -m\bigr) =\sup_{0\le s\le t}\{e_s+k^w_s\}.
\end{align}
Let $t\in[0,\zeta]$ be such that $e_t=M_t(e)$. Since $k^w$ is increasing, the right-hand side of \eqref{eq:max-identity-simplified} at such $t$ equals $a+k^w_t$. We derive $M_t(w)=w_t$ from \eqref{eq:max-identity-simplified} and \eqref{eq:deterministic-regulator}. The last statement follows since $e_t<a=e_{T_a(e)}$ for all $t<T_a(e)$.
\end{proof}

\begin{remark}
For the classical perturbed equation (with $x=m$ and without the reflection part $k^w$), there is a unique solution that can be written precisely as in \eqref{eq:max-phase} below. Then we have $M_t(w)-w_t = M_t(e)-e_t$, which implies that $w$ and $e$ attain their respective maxima at exactly
the same times.
For the reflected equation \eqref{eq:deterministic-regulator}, Corollary~\ref{cor:max} shows that $e_t=M_t(e)$ still implies $w_t=M_t(w)$, while the converse may fail 
because reflection at a boundary can cause $w$ to attain a new maximum even when $e$ is below its maximum. 
\end{remark}

Let now $\zeta:=\inf\{t>0: M_t(w) \vee m=b(t)\}$. For $m>0$, we refer to \cite[Proposition 6.2]{LeGall-Yor1990} and construct the solution of \eqref{eq:deterministic-regulator} up to time $\zeta$ uniquely as follows. We define the following maps: for every $t\ge 0$,
\begin{align}
    \Gamma^\m_t(e,x) &= x + e_t + \nu^*\sup_{0\le s\le t} e_s,\label{eq:max-phase}\\
   \Gamma^\b_t(e,b) &= e_t + \sup_{0\le s\le t} (b(s)-e_s)^+,\label{eq:boundary-phase}\\
   \fT^{(s)}_t(e) &= e(t+s)-e(s).\label{eq:translation}
\end{align}
We can deal with the case $x=m$ as follows. Let $\sigma^{(0)}:=0$ and $\sigma^{(1)}:=\inf\{t\ge 0: w_t=b(t)\}$. Since $b(0)=0$, we have $\sigma^{(1)}>0$ and therefore \eqref{eq:deterministic-regulator} has the unique solution $w_t = \Gamma^\m_t(e,x)$ for $t\in[0,\sigma^{(1)}]$.
Let $\sigma^{(2)}=\inf\{t\ge \sigma^{(1)}: w_t = M_{\sigma^{(1)}}(w)\}$. Then if $\sigma^{(1)}\ne \zeta$, we have $\sigma^{(2)}>\sigma^{(1)}$ and therefore can apply Skorokhod's lemma \cite[Lemma VI.2.1]{RevuzYor} to show that $w_{t} = b(\sigma^{(1)}) + \Gamma^\b_{t-\sigma^{(1)}}\big(\fT^{(\sigma^{(1)})}(e), \fT^{(\sigma^{(1)})}(b)\big)$ and $k^w_t = \big(w_{t}- b(\sigma^{(1)})\big) - \fT^{(\sigma^{(1)})}_{t-\sigma^{(1)}}(e)$ are the unique solution of \eqref{eq:deterministic-regulator} for $t\in[\sigma^{(1)},\sigma^{(2)}]$. 
We continue by alternating these two phases \eqref{eq:max-phase}--\eqref{eq:boundary-phase} whenever $\sigma^{(n)}<\zeta$, so
\begin{align*}
    w_t = \begin{cases}
        \Gamma^\m_{t-\sigma^{(2n)}}\big(\fT^{(\sigma^{(2n)})}(e), M_{\sigma^{(2n)}}(w)\big), & t\in[\sigma^{(2n)}, \sigma^{(2n+1)}\wedge\zeta);\\
        b(\sigma^{(2n+1)})+\Gamma^\b_{t-\sigma^{(2n+1)}}\big(\fT^{(\sigma^{(2n+1)})}(e),\fT^{(\sigma^{(2n+1)})}(b)\big), & t\in[\sigma^{(2n+1)}, \sigma^{(2n+2)}\wedge\zeta).
    \end{cases}
\end{align*}
If we had $\sigma^{(n)}<\zeta$ for all $n\ge 1$ while $\sigma_\infty:=\lim_{n\to \infty} \sigma^{(n)} <\infty$, we could apply Lemma~\ref{lem:identities} to extend the solution to time $\sigma_\infty$ and also obtain that this limit is $\zeta$, because $w_{\sigma^{(2n+1)}}=b(\sigma^{(2n+1)})$ and $w_{\sigma^{(2n+2)}}=M_{\sigma^{(2n+2)}}(w)=M_{\sigma^{(2n+1)}}(w)$ from the construction. 
Therefore this solution is unique up to time $\zeta$, and it is moreover nonanticipative, meaning that the restriction of $w$ to $[0,t]$ depends only on the driver $e$ and boundary $b$ on $[0,t]$. The case $0\le x<m$ can also be dealt with similarly: by Skorokhod's lemma, we have 
\begin{align}\label{eq:remark-step}
w_t = x+e_t+\sup_{0\le s\le t} (b(s)-x-e_s)^+ = e_t+\max\left\{x,\sup_{0\le s\le t}(b(s)-e_s)\right\}
\end{align}
until the first time when $w$ reaches $m$, and if this time is not $\zeta$, we can apply \eqref{eq:max-phase}--\eqref{eq:boundary-phase} alternately to construct the solution up to time $\zeta$. We call $\zeta$ the lifetime of the solution $w$.
\medskip

Let $(W,K)$ be a continuous $\rF$-adapted solution of \eqref{eq:deterministic-regulator} with $e=B$ and $0\le x\le m$ there and lifetime $\zeta$. The lifetime $\zeta$ is an $\rF$-stopping time from its definition. 
When $x=m$, \eqref{eq:deterministic-regulator} reduces to
\begin{equation}\label{eq:regulator}
\begin{aligned}
W_t&=(1-\nu) x+B_t+\nu M_t(W)+K_t,\qquad W_t\ge b(t),\qquad 0\le t< \zeta.
\end{aligned}
\end{equation}
For every $t\in [0,\zeta]$, define
\begin{equation}\label{eq:F-Y}
 F_t:= (M_t(W) \vee m)-b(t),
 \qquad Y_t:= (M_t(W)\vee m)-W_t,
\end{equation}
so $0\le Y_t\le F_t$. We call the two sets $\{t\in[0,\zeta):Y_t=0\}$ and $\{t\in [0,\zeta):Y_t=F_t\}$ the \emph{maximum face} and the \emph{boundary face} respectively. 
For an $\rF$-stopping time $\sigma$, on the event $\{\sigma<\zeta\}$, define
\begin{align}\label{eq:hitting-F}
    \tau_a:=\inf\{t\in(\sigma,\zeta): F_t =a\},
\end{align}
with $\inf\emptyset = \infty$ by convention.
The following result establishes a gambler-ruin estimate for the hitting time of $F$, which shows that, given a continuous boundary $b$ locally of finite variation that satisfies \eqref{eq:PB}, the corner $F=0$ is polar once $m$ is positive.

\begin{proposition}\label{prop:corner-power}
Fix $T,m>0$, $0\le x\le m$ and that $b$ satisfies \eqref{eq:PB}. 
There exists $\rho_T>0$ such that for every $\varepsilon>0$, $\rF$-stopping time $\sigma\le T$ and $0\le x\le m$, on the event $\{\sigma<\zeta,\,\varepsilon<F_\sigma<\rho_T\}$, 
\begin{equation}\label{eq:corner-probability}
 \bP(\tau_\varepsilon<\tau_{\rho_T},\ 
      \tau_\varepsilon\le T\mid\rF_\sigma)
 \le \left(\frac{4\varepsilon}{F_\sigma}\right)^2.
\end{equation}
In particular, when $e=B$, $\zeta=\infty$ a.s. and there exists a unique continuous and $\rF$-adapted pair $(W,K)$ solving \eqref{eq:deterministic-regulator} for all $t\ge 0$.
\end{proposition}

We start the proof of the proposition with the following lemma.

\begin{lemma}\label{lem:corner-annulus}
Fix $T,m>0,\, q\in(0,1/2)$ and suppose that $b$ satisfies \eqref{eq:PB}. There exists $\rho=\rho(T,q)>0$ such that for every $\rF$-stopping time $\sigma\le T$ and $0\le x\le m$, on the event $\{\sigma<\zeta,\, 0<F_\sigma<\rho\}$,
\begin{equation}\label{eq:one-annulus}
 \bP(\tau_{\frac12 F_\sigma}<\tau_{2F_\sigma}\mid\rF_\sigma)\le q.
\end{equation}
\end{lemma}

\begin{proof}
Let $p_0:=\bP\{\mathcal N(0,1)\ge 2\bar\nu+3\}>0$ and $J$ be an integer so that
$(1-p_0)^J\le q$. From \eqref{eq:PB}, we can take $\rho=\rho(T,q)$ so small that $J \rho^2\le 1$ and that $\omega^+_{b,T+1}(Jr^2)<r/2$ for all $0<r\le\rho$.
We observe that if $F_t=F_\sigma/2$ for some $t\in(\sigma,\zeta)$, then
\[
 b(t)-b(\sigma)= (M_t(W) \vee m) - (M_\sigma(W)\vee m) + F_\sigma-F_t\ge F_\sigma/2.
\]
Consequently, $\tau_{\frac12 F_\sigma}>\sigma+J F_\sigma^2$ on the event $\{\sigma<\zeta,\,0<F_\sigma<\rho\}$.

Write $s_j=\sigma+j F_\sigma^2$ for $j=0,\cdots,J$. They are $\rF_\sigma$-measurable random variables, and in particular $\rF$-stopping times. 
Recall $Y=(M(W)\vee m)-W$ in \eqref{eq:F-Y}.  From \eqref{eq:deterministic-regulator}, we have $B_t+K_t = \bar\nu((M_t(W) \vee m)-m)+(m-x)-Y_t$. Define the event 
\[
E_j:=\{\sigma<\zeta,\,0<F_\sigma<\rho\}\cap \{s_j<\zeta\} \cap \{\tau_{\frac12 F_\sigma}\wedge\tau_{2F_\sigma}>s_j\} \in \rF_{s_j}.
\]
Let us prove that $E_{j+1}\subseteq E_j\cap \{B_{s_{j+1}}-B_{s_j}< (2\bar\nu +3)F_\sigma\}$. On the event $E_j\cap \{s_{j+1}<\zeta\}$, we have $F_{s_j}\in(F_\sigma/2,2F_\sigma)$ and $Y_{s_j}\le F_{s_j}<2 F_\sigma$. If in addition $B_{s_{j+1}}-B_{s_j}\ge (2\bar\nu +3)F_\sigma$, we will have
\[
 \bar{\nu}\bigl((M_{s_{j+1}}(W)\vee m) - (M_{s_j}(W)\vee m)\bigr) \ge (B_{s_{j+1}}-B_{s_j}) - Y_{s_j} > (2\bar\nu+1)F_\sigma.
\]
But on this event, $b(s_{j+1})-b(s_j)<F_\sigma/2$ whenever that increment is positive.  Hence $F_{s_{j+1}}>2F_\sigma$ and thus $\tau_{2F_\sigma}<s_{j+1}$, which contradicts $E_{j+1}$ and proves the desired inclusion. 
Conditionally on $\rF_{s_j}$, the Brownian event $\{B_{s_{j+1}}-B_{s_j}\ge (2\bar\nu +3)F_\sigma\}$ has probability $p_0$ from the strong Markov property of $B$ at $s_j$. Thus $\bP(E_{j+1}\mid\rF_{s_j})\le (1-p_0)\ind_{E_j}$. Iteration over the $J$ blocks gives, on the event $\{\sigma<\zeta,\,0<F_\sigma<\rho\}$,
\begin{align}\label{eq:EJ}
 \bP\big(s_J<\zeta,\, \tau_{\frac12 F_\sigma}\wedge\tau_{2F_\sigma} > \sigma + J F_\sigma^2\mid\rF_\sigma\big)
 \le(1-p_0)^J\le q.
\end{align}
On the event $\{\sigma<\zeta,\,0<F_\sigma<\rho\}\cap\{\tau_{\frac12 F_\sigma}<\tau_{2F_\sigma}\}$, we have $\tau_{\frac12F_\sigma}<\infty$ and hence $s_J<\zeta,\, \tau_{\frac12 F_\sigma}>\sigma+J F_\sigma^2$ from the discussion in the first paragraph. 
It together with \eqref{eq:EJ} proves \eqref{eq:one-annulus}.
\end{proof}

\begin{proof}[Proof of Proposition~\ref{prop:corner-power}]
Fix $q=1/5$ and take $\rho_T=\rho$ in Lemma~\ref{lem:corner-annulus}. Let $k_* = \lceil \log_2(\rho_T/F_\sigma)\rceil$ and set $r:= 2^{-k_*} \rho_T$, thus $F_\sigma\in[r,2r)$. We shall work on the event 
\[A_k:=\{ \sigma<\zeta,\,\varepsilon<F_\sigma<\rho_T\}\cap \{k_*=k\}\in\rF_\sigma\]
for every integer $k\ge 1$, so $r$ is a deterministic constant on each $A_k$. If $\varepsilon\ge r$, then $\varepsilon/F_\sigma \ge1/2$ hence \eqref{eq:corner-probability} is trivial. We now assume $\varepsilon<r$. Let $n=\lfloor \log_2(r/\varepsilon)\rfloor$ and $r_0=2^{-n}r\in[\varepsilon,2\varepsilon)$. 
On the event $A_k$, we define $\sigma_0 := \inf\{t\ge \sigma: F_t\in\{r,2r\}\}\wedge T\wedge\zeta$ and then $Z_0:=\log_2(F_{\sigma_0}/r)$ if $\sigma_0<\zeta$ and $F_{\sigma_0}\in\{r,2r\}$ otherwise $\infty$. We then recursively define, for every $j\ge 0$, if $Z_j=\infty$ or $\sigma_j\ge T\wedge\zeta$, set
\[
\sigma_{j+1}:=\sigma_j,\qquad Z_{j+1}:=\infty;
\]
otherwise
\begin{align*}
\sigma_{j+1} &:= \inf\{t>\sigma_j:F_t\in\{2^{Z_j-1}r,2^{Z_j+1}r\}\}\wedge T\wedge\zeta, \\
Z_{j+1} &:= \begin{cases}
Z_j+1, & \text{if}\quad  \sigma_{j+1}<\zeta ,\,F_{\sigma_{j+1}} = 2^{Z_j+1}r;\\
Z_j-1, & \text{if}\quad  \sigma_{j+1}<\zeta,\, F_{\sigma_{j+1}} = 2^{Z_j-1}r;\\
\infty, & \text{otherwise}.
\end{cases}
\end{align*}
Set $\widehat\tau = \inf\{j\ge 0: Z_j\in\{-n,k,\infty\}\}.$
By Lemma~\ref{lem:corner-annulus}, we have on the event $A_k\cap \{\widehat\tau>j\}$,
\[
\bP(Z_{j+1} = Z_j-1\mid \rF_{\sigma_j} ) \le q =\frac15.
\]
Therefore, with $\theta=q/(1-q)=1/4$, and since $x\mapsto x\theta^{-1}+(1-x)\theta$ is increasing, we deduce that for all $j\ge 0$,
\[
\bE[\theta^{Z_{(j+1)\wedge\widehat\tau}}\mid \rF_{\sigma_j}] \le \theta^{Z_{j\wedge\widehat\tau}} (q\theta^{-1} + (1-q)\theta) = \theta^{Z_{j\wedge\widehat\tau}},
\]
with $\theta^\infty:=0$ by convention.
We now observe that $\tau_\varepsilon<\tau_{\rho_T},\, \tau_\varepsilon\le T$ implies $Z_{\widehat\tau} = -n$ and hence, on the event $A_k$,
\begin{align*}
&\bP(\tau_\varepsilon<\tau_{\rho_T},\ \tau_\varepsilon\le T\mid\rF_{\sigma_0}) 
\le \bP(Z_{\widehat\tau}=-n\mid\rF_{\sigma_0})\le \theta^n\bE(\theta^{Z_{\widehat{\tau}}}\mid\rF_{\sigma_0}) \le \theta^{n+Z_0} \le \theta^n\\
=\, & 2^{-2n} \le \left(\frac{2\varepsilon}{r}\right)^2 \le \left(\frac{4\varepsilon}{F_{\sigma}}\right)^2,
\end{align*}
since $\varepsilon\le 2^{-n} r<2\varepsilon$ and $r>F_\sigma/2$.
Conditioning on $\rF_\sigma$ and summing over the disjoint events $A_k$ over $k$ prove \eqref{eq:corner-probability}.

\medskip

We now prove that $\zeta=\infty$ a.s. Suppose for the sake of contradiction that $\zeta<\infty$ with positive probability. By Lemma~\ref{lem:identities}, the solution extends continuously to $\zeta$. Recall $F_\zeta=0$ by definition. Fix $T>0$. On the event $\{\zeta\le T\}$, there exists some rational time $r<\zeta$ such that $F_r\in (0,\rho_T)$ and $F_t<\rho_T$ for all $r\le t< \zeta$. Applying \eqref{eq:corner-probability} with $\sigma=r$ for all $r<T$ rational and sending $\varepsilon\downarrow 0$, we deduce that $\bP(F_\zeta=0,\,\zeta\le T)=0$. Sending $T\to \infty$ and combining $F_\zeta=0$ imply a contradiction. Pathwise uniqueness and adaptedness of the solution $(W,K)$ follow from the alternating construction shown at the beginning of this section.
\end{proof}

\begin{remark}\label{rmk:rho}
In Section~\ref{sec:cycles}, we will deal with the solution of \eqref{eq:deterministic-regulator} starting from any time $s\le T$.
As shown in the proof of Lemma~\ref{lem:corner-annulus}, the threshold $\rho_T$ in Proposition~\ref{prop:corner-power} is chosen to ensure that $J \rho^2\le 1$ and that $\omega^+_{b,T+1}(Jr^2)<r/2$ for all $0<r\le\rho$, where $J$ is an integer so that $(1-\bP\{\mathcal N(0,1)\ge 2\bar\nu+3\})^J\le 1/5$. Therefore the choice of $\rho_T$ still works when applying Proposition~\ref{prop:corner-power} to this kind of solution.
\end{remark}

We conclude this section with a simple observation.

\begin{proposition}\label{prop:positive-gap}
Fix $x\ge 0$ and assume that $b$ is continuous and locally of finite variation and satisfies \eqref{eq:PB}. Let $(W,K)$ be $\rF$-adapted processes solving \eqref{eq:regulator}. Then
\begin{equation}\label{eq:K-is-local-time}
 K_t=\frac12L_t^0(W-b),\qquad t\ge0.
\end{equation}
In particular, Theorem~\ref{thm:main} holds for $x>0$.
\end{proposition}

\begin{proof}
Put $X=W-b$ and fix $T\ge 0$.  We first prove
\begin{equation}\label{eq:db-contact-zero}
 \int_0^T\ind_{\{X_t=0\}}\,|\dd b|(t)=0\qquad\text{a.s.}
\end{equation}
Fix a deterministic $t>0$. On the event $\{X_t=0\}$, we obtain from \eqref{eq:regulator} that for $0<h<t$, since $K$ is increasing and $X_{t-h}\ge 0$,
\[
0 = X_t \ge  (B_t-B_{t-h}) - (b(t)-b(t-h)) + \nu (M_t(W)-M_{t-h}(W)).
\]
Recall $F=M(W)-b$.
If $F_t> 0$, then $M(W)$ is constant on $[t-h,t]$ for all $h>0$ sufficiently small; if $F_t=0$, then $F_{t-h}\ge 0$ implies $M_t-M_{t-h}\le b(t)-b(t-h)$.
We conclude that on the event $\{X_t=0\}$, for all sufficiently small $h>0$,
\[
B_t-B_{t-h}\le \max\{1,\bar\nu\}(b(t)-b(t-h))^+.
\]
However, for every fixed $t>0$, the law of the iterated logarithm gives almost surely
\[
 \limsup_{h\downarrow0}
 \frac{B_t-B_{t-h}}{\sqrt{2h\log\log(1/h)}}=1,
\]
while \eqref{eq:PB} gives $(b(t)-b(t-h))^+=o(\sqrt h)$, yielding a contradiction. We deduce that $ \bP(X_t=0)=0$ for every fixed $t> 0$. Since $b$ is continuous, the measure $|\dd b|$ has no atom at zero.  Equation \eqref{eq:db-contact-zero} now follows by Fubini's theorem.

Since $b$ is locally of finite variation, $X$ is a continuous semimartingale with $B$ as the local martingale part. Since $F=M(W)-b$ is continuous and locally of finite variation, we have $\ind_{\{F=0\}}\dd M(W) = \ind_{\{F=0\}}\dd b$. Recall that $\dd M(W)$ is supported on $\{M(W)=W\}$. Together with \eqref{eq:db-contact-zero}, this shows that $\dd M(W)$ does not charge $\{X=0\}$. 
We also have\( \int_0^t\ind_{\{X_s=0\}}\dd B_s=0,\) because the stochastic integral has quadratic variation $\int_0^t\ind_{\{X_s=0\}}\dd s=0$ by the occupation times formula \cite[Corollary VI.1.6]{RevuzYor}.
Since $\dd K$ is carried by $\{X=0\}$, we conclude that
\[
 \int_0^t\ind_{\{X_s=0\}}\dd X_s=K_t.
\]
On the other hand, we apply Tanaka's formula to $X=X^+$ and thus derive $\int_0^t\ind_{\{X_s=0\}}\dd X_s=L_t^0(X)/2$. This proves \eqref{eq:K-is-local-time}. The second result follows by Proposition~\ref{prop:corner-power} and the identification \eqref{eq:K-is-local-time}.
\end{proof}

\section{Comparison of two solutions}
\label{sec:comparison}
This section provides results analogous to \cite[Lemmas 2--3]{ChaumontDoney} in the time-dependent-boundary setting. The proofs follow their strategy; we present the details needed in the time-dependent boundary case for completeness.

Assume $\nu\in [1/2,1)$ throughout this section. Fix $T\ge 0$. Let $w$ be the solution of \eqref{eq:deterministic-regulator} with $m>0$, $x\in [0,m]$.
We say that the faces of $w$ are separated on some interval $I$ if
$M_t(w) \vee m>b(t)$ there. 
Suppose $w$ starts from $x$ such that the faces of $w$ are separated on $[0,T]$\footnote{We are only concerned about the case where $e$ is a typical sample path of a Brownian motion, so the lifetime $\zeta=\inf\{t>0:M_t(w) \vee m=b(t)\}$ is infinite by Proposition~\ref{prop:corner-power} and faces of $w$ are separated on $[0,\infty)$.}. 
We refer to \cite[Section 2]{ChaumontDoney} and define the completed \emph{round-trip} of $w$: intuitively speaking, a round-trip of $w$ starts from the maximum face, visits the boundary face and returns to the maximum face.
Thus it can be represented by times $t_1<t_2<t_3$ such that
\[
w_{t_1} = M_{t_1}(w)\ge m,\qquad w_{t_2}=b(t_2),\qquad w_{t_3}=M_{t_3}(w)=M_{t_2}(w).
\]
Let $N_w(T)$ be the maximal number of completed round-trips with mutually disjoint interiors contained in $[0,T]$. Under the separated-face assumption on $w$ starting from $x>0$ on the time interval $[0,T]$, we have $\inf_{0\le s\le T}\bigl((M_s(w) \vee m)-b(s)\bigr)>0$; together with the uniform continuity of $w$, we obtain that $N_w(T)<\infty$.

The following is an analog of \cite[Lemma 3]{ChaumontDoney}, 
which bounds the distance between two solutions in terms of the round-trips completed by the solution with the lower initial state.
 
\begin{proposition}\label{prop:comparison}
Let $T,r,d\ge 0$. Let $w$ solve \eqref{eq:deterministic-regulator} with initial position $m=x=r$, and let $z$ be a solution with initial position $m=x=r+d$ and the same $e$ and $b$.  Suppose that both pairs of faces are separated on $[0,T]$.
Then
\begin{equation}\label{eq:comparison-bound}
 \sup_{0\le t\le T}|z_t-w_t|
 \le d\,(\nu^*)^{2N_w(T)+1}.
\end{equation}
\end{proposition}

The discussion at the beginning of Section~\ref{sec:positive} shows that both $z$ and $w$ have lifetime greater than $T$ in the setting of Proposition~\ref{prop:comparison}. We start the proof with a lemma.

\begin{lemma}\label{lem:lower-regulator-order}
Let $w,z$ satisfy the conditions in Proposition~\ref{prop:comparison}.  Restricted to any interval $I$ on which $z\ge w$, the difference between the measures $\dd k^w$ and $\dd k^z$ is a nonnegative measure.
\end{lemma}

\begin{proof}
From \eqref{eq:deterministic-regulator}, the difference $D=z-w$ is continuous, nonnegative, and of finite
variation on $I$. Applying Tanaka's formula to $D^-=0$ gives $\ind_{\{D=0\}}\dd D=0.$ Put
\[
 A:=\{t\in I:D_t=0,\ w_t=z_t=b(t)\}.
\]
The measure $\dd k^z$ is carried by $\{z=b\}$. We have $z\ge w\ge b$ on $I$, which implies that $\dd k^z$ is carried by $A$. Moreover, from the separated-face assumption, we obtain that neither $\dd M(w)$ nor $\dd M(z)$ charges $A$.
Therefore \( \ind_A\dd D=\ind_A(\dd k^z-\dd k^w)\) as signed measures, and thus
\[
 \dd k^z=\ind_A\dd k^z=\ind_A\dd k^w\le\dd k^w.
\]
This completes the proof.
\end{proof}

\begin{lemma}\label{lem:crossover-face}
Let $w,z$ satisfy the conditions in Proposition~\ref{prop:comparison}.
Suppose $z_s\ge w_s$ for some $s\ge 0$ and define $C:=\inf\{t>s:z_t<w_t\}$. If $C\le T$, then 
\begin{equation}\label{eq:w-max-at-crossing}
 w_C=z_C=M_C(w)<M_C(z).
\end{equation}
Define $C':=\inf\{t>C:z_t=M_C(z)\}$. Then $z\le w$ on $[C,C'\wedge T]$.
\end{lemma}

\begin{proof}
Continuity of $w$ and $z$ gives $w_C=z_C$. We have $w_C\ne b(C)$; otherwise $M_C(w),M_C(z)>b(C)=w_C=z_C$ from the separated-face condition, and hence \eqref{eq:boundary-phase} implies that $z=w$ on $[C,C+\eta]$ for some $\eta>0$. We now prove that $M_C(w)=w_C$; if $M_C(w)>w_C$, then $\dd M(w),\dd k^w,\dd k^z\equiv 0$ on $[C,C+\eta]$ for some $\eta>0$. Therefore $w_t-z_t = -\nu (M_t(z)-M_C(z))\le 0$ for $t\in[C,C+\eta]$ from \eqref{eq:deterministic-regulator}, contradicting the definition of $C$. If $M_C(z)=z_C$, then \eqref{eq:max-phase} implies that $z=w$ on $[C,C+\eta]$ for some $\eta>0$, again a contradiction. It completes the proof of \eqref{eq:w-max-at-crossing}.

We now prove the second statement. We have $M(z)\equiv M_C(z)$ in $[C,C'\wedge T]$. From \eqref{eq:deterministic-regulator}, we have for all $t\in [C,C'\wedge T]$,
\[
D_t:=z_t-w_t = -\nu (M_t(w)-M_C(w)) + (k^z_t-k^z_C) - (k^w_t-k^w_C).
\]
We have $\ind_{\{D>0\}}\dd k^z \equiv 0$ since $z>w\ge b$ on $\{D>0\}$. Therefore, from Tanaka's formula, for all $t\in [C,C'\wedge T]$,
\[
0\le D^+_t = -\nu\int_C^t \ind_{\{D_s>0\}}\dd M_s(w) - \int_C^t \ind_{\{D_s>0\}}\dd k^w_s \le 0,
\]
and hence $D^+ \equiv 0$, i.e. $z\le w$ on $[C,C'\wedge T ]$.
\end{proof}

\begin{proof}[Proof of Proposition~\ref{prop:comparison}]
The case $d=0$ is immediate because of the uniqueness of the solution of the equation \eqref{eq:deterministic-regulator}. Assume henceforth $d>0$.
Using the notation in Lemma~\ref{lem:crossover-face}, we define $s=0$, $C_0=C'_0=0,\,C_1=C,\, C'_1=C'$ and then $C_2:=\inf\{t>C_1: z_t>w_t\}\ge C'_1$, $C'_2:=\inf\{t>C_2: w_t = M_t(w)\}$, etc. Then for every $t\in [0,T]$, we have 
\[
M_t(z) - M_t(w) \le \max\left\{d, \sup_{0\le s\le t}(z_s-w_s) \right\}.
\]
By Lemma~\ref{lem:lower-regulator-order}, we have $k^z-k^w\le 0$ on $[0,C_1\wedge T]$. Therefore, from \eqref{eq:deterministic-regulator}, 
\begin{align}\label{eq:comparison-middle-step-1}
z_s-w_s \le \bar\nu d + \nu \max\left\{d, \sup_{0\le s\le t}(z_s-w_s) \right\},\qquad 0\le s\le t\le C_1\wedge T,
\end{align}
and therefore $z-w\in [0,d]$ on $[0,C_1\wedge T]$.

Suppose $C_1<T$. Then by Lemma~\ref{lem:crossover-face}, we have $w_{C_1}=z_{C_1}=M_{C_1}(w)$ and $h:=M_{C_1}(z)-z_{C_1}> 0$. We observe that $z\le w$ on $[C_1,C_2\wedge T]$ and therefore $(k^w-k_{C_1}^w) - (k^z-k_{C_1}^z)\le 0$ on this time interval. Therefore, from \eqref{eq:deterministic-regulator},  we obtain that for $C_1\le s\le C_2\wedge T$,
\begin{align}\label{eq:comparison-middle-step-2}
w_s-z_s \le\nu\{(M_s(w)-M_{C_1}(w))-(M_s(z)-M_{C_1}(z))\} =\nu\{M_s(w)-M_s(z)+h\}.
\end{align}
Let $D^*_t:=\sup_{C_1\le s\le t} (w_s-z_s)$ for $C_1\le t\le C_2\wedge T$. For $C_1\le s\le t$, we have $M_{C_1}(w)<M_{C_1}(z)\le M_s(z)$ and $w_u\le z_u +D^*_t$ for $u\in [C_1,s]$. Thus $M_s(w)-M_s(z)\le D^*_t$, and \eqref{eq:comparison-middle-step-2} implies that $D^*_t \le \nu (D^*_t+h)$ hence $D^*_t\le \nu^* h\le \nu^*d$. 
Moreover, if $C_2<T$, then by Lemma~\ref{lem:crossover-face}, $w_{C_2}=z_{C_2}=M_{C_2}(z)$ and $h':= M_{C_2}(w)-w_{C_2}>0$. We have $h'\le D^*_{C_2}\le \nu^* d$ since $M_{C_1}(w)< M_{C_1}(z)$ and $w\le z + D^*$ on $[C_1,C_2]$.
Interchanging $w$ and $z$ gives \(\sup_{C_2\le s\le C_3\wedge T}(z_s-w_s)\le d(\nu^*)^2.\)
We therefore deduce by induction that, 
\begin{align}\label{eq:comparison-bound-middle-step}
\sup_{0\le s\le T} |z_s-w_s| \le d(\nu^*)^{\theta_T+1},
\end{align}
where $\theta_T:=\sup\{n\ge 0: C_n\le T\}$.

Let $T_{n}:=\inf\{t>C_{2n}:w_t=b(t)\}$ for $n\ge 0$. We now prove that $T_{n}\le C_{2n+1}$ if $C_{2n+1}\le T$; thus by Lemma~\ref{lem:crossover-face}, a round-trip of $w$ is completed in $[C_{2n-1},C_{2n+1}]$ (where $C_{-1}:=0$).
For $n\ge 0$, the explicit formula \eqref{eq:max-phase} for $z$ and $w$ yields that $z-w=d>0$ on $[0,T_0]$ and hence $T_0<C_1$. We now suppose $n\ge 1$. By Lemma~\ref{lem:crossover-face} and the separated-face assumption, we have $z_{C_{2n}}=w_{C_{2n}}=M_{C_{2n}}(z)>b(C_{2n}),\, M_{C_{2n}}(w)=M_{C_{2n}}(z)+h$ for some $h>0$. 
Suppose $T_n>C_{2n+1}$. On the time interval $[C_{2n},C_{2n+1}]$, since $z\ge w$, $z$ and $w$ are both greater than $b$. Therefore, for every $t\in [C_{2n},C_{2n+1}\wedge T_n]$, we can write the explicit form (similar to \eqref{eq:max-phase}) of the solution as
\begin{align*}
w_t &= w_{C_{2n}} + (e_t - e_{C_{2n}}) + \nu^* \left[\sup_{C_{2n}\le s\le t} (e_s - e_{C_{2n}}) -h \right]^+, \\
z_t &= z_{C_{2n}} + (e_t - e_{C_{2n}}) + \nu^* \sup_{C_{2n}\le s\le t} (e_s - e_{C_{2n}}).
\end{align*}
Therefore,
\[
z_t - w_t = \nu^* \min\left\{\sup_{C_{2n}\le s\le t} (e_s - e_{C_{2n}}),h \right\} \ge 0.
\]
But by Lemma~\ref{lem:crossover-face} again, $w_{C_{2n+1}} = z_{C_{2n+1}} = M_{C_{2n+1}}(w)$, which yields that $e_t \le e_{C_{2n}}$ for all $t\in [C_{2n},C_{2n+1}]$. Thus
\[
w_{C_{2n+1}} =  w_{C_{2n}} + (e_{C_{2n+1}} - e_{C_{2n}}) \le w_{C_{2n}} = M_{C_{2n}}(z) < M_{C_{2n}}(w) \le M_{C_{2n+1}}(w) = w_{C_{2n+1}},
\]
a contradiction. Therefore, a round-trip of $w$ is completed in $[C_{2n-1},C_{2n+1}]$ for all $n$ such that $C_{2n+1}\le T$, which implies $\theta_T\le 2 N_w(T)$. Combining this bound with \eqref{eq:comparison-bound-middle-step} proves \eqref{eq:comparison-bound}.
\end{proof}


\section{A logarithmic bound on the number of round-trips}\label{sec:cycles}

Recall the definition of $N_w(T)$ in Section~\ref{sec:comparison}, i.e. the  number of round-trips in $[0,T]$ of the solution $w$ of \eqref{eq:deterministic-regulator}. We will prove that, if the driving function $e$ in \eqref{eq:deterministic-regulator} is a Brownian motion $B$, then as $x\downarrow 0$, the number of round-trips $N_w(T)$ admits a logarithmic upper bound with high probability. 

Recall the translation operator $\fT^{(s)}$ in \eqref{eq:translation} for $s\ge 0$. For $m>0$ and $x\in[0,m]$, let $\bP^m_{s,x}$ denote the law of $(W_{u+s}:u\ge 0)$ defined by setting $W_{u+s} = w_u+b(s)$, where $w$ is the solution of \eqref{eq:deterministic-regulator} with $e\sim B$ and the boundary $\fT^{(s)}(b)$, 
and abbreviate $\bP_{s,x}=\bP^x_{s,x}$. The process $W$ under $\bP^m_{s,x}$ can be understood as the solution starting from time $s$ and position $x+b(s)$, with maximum face at least $m+b(s)$. 
Let $\fN_T$ denote the number of completed round-trips of $W$ in $[s,T]$. We first have the following result. 

\begin{proposition}\label{prop:positive-width-tightness}
Fix $T,m>0$ and suppose that $b$ satisfies \eqref{eq:PB}.  Then for any $\varepsilon>0$, there exists some $n_\varepsilon\in\bN$ such that
\[
\sup_{0\le s\le T,\, 0\le x\le m}\bP^m_{s,x}\left(\fN_T> n_\varepsilon \right) \le \varepsilon.
\]
\end{proposition}

Proposition~\ref{prop:positive-width-tightness} can be understood as follows. For any $s\in[0,T]$, $m>0$ and $x\in[0,m]$, consider the solution starting from time $s\in [0,T]$ and position $x+b(s)$ with maximum face at least $m+b(s)$. Then the collection of numbers of completed round-trips before time $T$ is tight.

\begin{proof}[Proof of Proposition~\ref{prop:positive-width-tightness}]
For $0\le s\le T$ and $0\le x\le m$, we realize $W\sim\bP^m_{s,x}$ using the driving function $\fT^{(s)}(B)$ and extend $W$ by setting $W_u=b(s)+x$ for $0\le u\le s$. 
Put $F_t=M_t(W)\vee (m+b(s))-b(t)$ for $t\ge s$. We first obtain a lower bound on $F$ with high probability that does not depend on $s$ and $x$. Let $\rho_T$ be given by Proposition~\ref{prop:corner-power} and choose \(0<\eta<\rho<\min\{m,\rho_T/2\}.\) 
Recall $\tau_a = \inf\{t>s:F_t=a\}$ in \eqref{eq:hitting-F} for $a\ge 0$. Let $\tau^{(0)}_\rho=\tau_\rho$ and recursively, for $k\ge 1$,
\[
\tau^{(k-1)}_{\rho_T}:=\inf\{t>\tau^{(k-1)}_\rho:F_t=\rho_T\},\qquad \tau^{(k)}_\rho:=\inf\{t>\tau^{(k-1)}_{\rho_T}:F_t=\rho\},
\]
with $\inf\emptyset=\infty$. Then we have
\begin{align}\label{eq:decomp-inf F}
\left\{\inf_{s\le t\le T} F_t\le\eta \right\} = \bigcup_{k\ge 0} \{\tau^{(k)}_{\rho}< \tau_\eta<\tau^{(k)}_{\rho_T},\, \tau_\eta\le T \},
\end{align}
since every visit of $F$ to a level $\eta<\rho$ must be preceded by a visit to $\rho$. For every $k\ge 1$, on the event $\{\tau^{(k)}_{\rho}< \tau_\eta\le T\}$, we have in particular $\tau^{(k-1)}_{\rho_T}< \tau^{(k)}_\rho\le T$ and hence
\[
 b(\tau^{(k)}_\rho)-b(\tau^{(k-1)}_{\rho_T}) = \left(M_{\tau^{(k)}_{\rho}}(W)\vee \big(m+b(s)\big)\right) - \left(M_{\tau^{(k-1)}_{\rho_T}}(W)\vee \big(m+b(s)\big)\right) + \rho_T-\rho 
 \ge \rho_T-\rho.
\]
Recall $\omega_{b,T}^+(h)$ in \eqref{eq:positive-modulus}. Continuity of $b$ supplies $h_b>0$ such that $\omega^+_{b,T}(h_b)<\rho_T-\rho$. Thus for every $k\ge 1$, $\tau^{(k)}_\rho-\tau^{(k-1)}_{\rho_T}>h_b$ on the event $\{\tau^{(k)}_{\rho}< \tau_\eta\le T\}$. Therefore, thanks to Remark~\ref{rmk:rho}, we can apply Proposition~\ref{prop:corner-power} for $F$ at $\tau^{(k)}_{\rho}$ for $k=0,\cdots, \lfloor (T-s)/h_b\rfloor$ and, combining these bounds with \eqref{eq:decomp-inf F}, deduce that for every $0\le s\le T$ and $0\le x\le m$,
\begin{equation}\label{eq:positive-width-minimum}
 \bP_{s,x}^m\left(\inf_{s\le t\le T} F_t\le\eta \right)
 \le \left(1+\frac{T}{h_b}\right)\left(\frac{4\eta}\rho\right)^2.
\end{equation}

Fix $\varepsilon>0$ and choose $\eta$ so that the right-hand side of \eqref{eq:positive-width-minimum} is smaller than $\varepsilon/2$.  By the uniform continuity of $b$ and $B$ on $[0,T]$, we can choose a deterministic $h_0>0$ such that for all $s\le T$,
\begin{align}
 &\sup_{u,v\in[0,T],\,|u-v|\le h_0}
 |b(v)-b(u)|<\frac\eta4,\label{eq:absolute-boundary-modulus} \\
 \bP\bigg( &\sup_{\substack{u,v\in[0,T],\,|u-v|\le h_0}} |B_v-B_u|\ge \frac\eta4\bigg)<\frac\varepsilon2.\label{eq:brownian-modulus-event}
\end{align}
Consider the event $\bar A$ as the complement of the union of events appearing in \eqref{eq:positive-width-minimum} and \eqref{eq:brownian-modulus-event}. Then we have $\bP(\bar A)\ge 1-\varepsilon$. 
Let $[t_1,t_3]\subseteq[s,T]$ be the time interval on which $W$ completes a round-trip presented at the beginning of Section~\ref{sec:comparison}, and choose $t_2\in(t_1,t_3)$ to be the first time when $W$ reaches the boundary, so $M_{t_1}(W)=W_{t_1}$, $M_{t_2}(W)=M_{t_3}(W)=W_{t_3}$ and $W_{t_2}=b(t_2)$. Then on the event $\bar A$, we have from \eqref{eq:max-phase}
\[
\eta\le F_{t_2} =\sup_{t_1\le u\le t_2}(B_u-B_{t_1}) - (B_{t_2}-B_{t_1}),
\]
and therefore $t_2-t_1\ge h_0$. On the other hand, \eqref{eq:boundary-phase} gives
\[
\eta\le F_{t_3}=W_{t_3}-b(t_3)=\Gamma^\b_{t_3-t_2}(\fT^{(t_2)}(B),\fT^{(t_2)}(b)) -\fT^{(t_2)}_{t_3-t_2}(b).
\]
Using \eqref{eq:absolute-boundary-modulus} and the Brownian modulus
bound on $\bar A$, we obtain that $t_3-t_2\ge h_0$ on $\bar A$. Therefore $\fN_T\le 1+T/(2h_0)$ for all $x\in[0,m]$ and $s\le T$.  This completes the proof of the asserted tightness.
\end{proof}

The main result of this section is as follows.

\begin{proposition}\label{prop:round-trip-count}
Fix $T>0$ and suppose that $b$ satisfies \eqref{eq:PB}. 
 Then for every $A>\bar\nu$,
\begin{equation}\label{eq:round-trip-probability}
 \limsup_{x\downarrow 0} \sup_{0\le s\le T}\bP_{s,x}(\fN_T>A\log(1/x)) = 0.
\end{equation} 
\end{proposition}

The remainder of this section is devoted to the proof of Proposition~\ref{prop:round-trip-count}.
Fix a deterministic horizon $T>0$ and an \emph{outer cap}
\begin{equation}\label{eq:outer-cap}
 0<\bar\rho<\rho_{T},
\end{equation}
where $\rho_{T}$ is supplied by Proposition~\ref{prop:corner-power}. 
Consider the process $W\sim\bP_{s,x}$ in Proposition~\ref{prop:round-trip-count} starting at time $s$ and position $b(s)+x$, so $F_s=M_s(W)-b(s)=x<\bar\rho$. We define, under $\bP_{s,x}$, 
\begin{align}
\sigma^\downarrow &:= \inf\{t>s: W_t=b(t)\} \label{eq:sigma_sf}\\
\sigma^\uparrow & := \tau_{\bar\rho} \wedge \inf\{t>\sigma^\downarrow:W_t = M_{\sigma^\downarrow}(W)\}, \label{eq:tau_sf}\\
\widehat\sigma^\uparrow &:= \sigma^\uparrow\wedge T,\qquad \xi^L := \begin{cases} \log(F_{\sigma^\uparrow}/x)\wedge L, & \sigma^\uparrow\le T,\\
L, &\sigma^\uparrow> T.
\end{cases}\label{eq:tau-xi-truncated}
\end{align}
We call the trajectory of $W$ in $[s,\sigma^\uparrow]$ a \emph{cycle with an outer cap} $\bar\rho$, which is either a round-trip of $W$ or a path stopped at $\bar\rho$.

\begin{lemma}\label{lem:long-bad-cycle}
Fix $L>0$ and let $\widehat\sigma^\uparrow$ and $\xi^L$ be defined in \eqref{eq:tau-xi-truncated}. Then 
\begin{equation}\label{eq:long-bad-cycle}
 \lim_{Q\to\infty}\limsup_{x\downarrow0} \sup_{0\le s\le T} \bP_{s,x}(\widehat\sigma^\uparrow-s> Qx^2,\ \xi^L<L)=0.
\end{equation}
\end{lemma}

\begin{proof}
Choose \(J>1\) and put
\[
\underline{x}:=\ee^{-J}x,\qquad x_L:=\ee^L x,\qquad \overline{x}:=\ee^{L+J} x.
\]
Choose \(\eta_J>0\) sufficiently small so that \(\overline{x}<\bar\rho\) for every $0<x\le \eta_J$. Define the event
\[
E_{J,x} := \left\{ \underline{x}<F_t<\overline{x} \text{ for every }t\in[s,\widehat \sigma^\uparrow] \right\}.
\]
We first control the duration of the cycle on the event \(E_{J,x}\). On this event, the outer cap is not reached since \(\overline{x}<\bar\rho.\)
Recall $Y:=M(W)-W$ in \eqref{eq:F-Y}. Then $Y_s=0$ and $Y_{t}\le F_t<\overline{x}$ for $s\le t\le \widehat\sigma^\uparrow$  on the event \(E_{J,x}\). From \eqref{eq:max-phase}, we have
\begin{align}\label{eq:E_Jf-before-sigma}
 Y_{s+u} = \sup_{0\le v\le u}\fT^{(s)}_v(B) - \fT^{(s)}_u(B),\qquad 0\le u\le (\sigma^\downarrow\wedge T)-s.
\end{align}
If furthermore $\sigma^\downarrow<T$, then on $[\sigma^\downarrow,\sigma^\uparrow\wedge T]$, since $M$ remains constant while $K$ is increasing, \eqref{eq:regulator} implies
\begin{align}\label{eq:E_Jf-after-sigma}
W_{\sigma^\downarrow+u} \ge b(\sigma^\downarrow)+\fT^{(\sigma^\downarrow)}_u(B), \qquad 0\le u\le \widehat\sigma^\uparrow-\sigma^\downarrow,
\end{align}
with equality at $u=0$. 
For a continuous process $Z$ starting at 0, write $T_Z(a):=\inf\{t>0:Z_t=a\}$. Recall that $M(Z)$ is the running supremum process of $Z$. We deduce that, uniformly in $0\le s\le T$ and $0<x\le \eta_J$,
\begin{align*}
\bP_{s,x}(\widehat\sigma^\uparrow - s>Qx^2, E_{J,x}) &\le \bP_{s,x}(\sigma^\downarrow\wedge T - s>\tfrac12 Qx^2, E_{J,x}) + \bP_{s,x}(\sigma^\uparrow\wedge T - \sigma^\downarrow\wedge T>\tfrac12 Qx^2, E_{J,x})\\
&\le \bP(T_{M(B)-B}(\overline{x})\ge \tfrac12 Qx^2) + \bP(T_B(\overline{x})\ge \tfrac12 Qx^2)\\
&\le \bP\left(T_{|B|}(1)\ge \frac{Q}{2\ee^{2(L+J)}}\right) + \bP\left(T_{B}(1)\ge \frac{Q}{2\ee^{2(L+J)}}\right)\to 0 \quad \text{as}~ Q\to\infty,
\end{align*}
where the second inequality follows by the strong Markov property of $B$ and \eqref{eq:E_Jf-before-sigma}--\eqref{eq:E_Jf-after-sigma}, and the last one by  L\'evy's identity and the scaling property of $B$.

It remains to control the complement of $E_{J,x}$ on $\{\xi^L<L\}$.   Recall $\tau_a:=\inf\{t>\sigma:F_t=a\}$ in \eqref{eq:hitting-F} where we take $\sigma\equiv s$. Proposition~\ref{prop:corner-power} gives
\[
\bP_{s,x}(\tau_{\underline{x}}\le \widehat\sigma^\uparrow)\le \left(\frac{4\underline{x}}{x}\right)^2 = 16 \ee^{-2J}.
\]
Now consider the event $A_+:=\{\tau_{\overline{x}}<\widehat\sigma^\uparrow,\ \xi^L<L\}.$ Then $\widehat\sigma^\uparrow = \sigma^\uparrow\le T$ and $F_{\sigma^\uparrow}<\ee^L x = x_L$ on this event. Since \(\overline{x}<\bar\rho\), we have \( \log(\bar\rho/x)>L+J>L.\) Recalling \eqref{eq:tau-xi-truncated}, we deduce that on the event $A_+$, the capped cycle cannot terminate by reaching the outer cap $\bar\rho$. 
On the other hand, \(F_{\tau_{\overline{x}}}=\overline{x}=\ee^{L+J}x>x_L\) if $\tau_{\overline{x}}<\infty$. Therefore on the event \(A_+\), the process $F$ reaches \(x_L\) at some time after \(\tau_{\overline{x}}\) and before \(\widehat\sigma^\uparrow\wedge \tau_{\bar\rho}\). Applying Proposition~\ref{prop:corner-power} conditionally at
\(\tau_{\overline{x}}\wedge T\) gives
\[
\bP_{s,x}(A_+\mid\rF_{\tau_{\overline{x}}\wedge T}) = \bP_{s,x}(A_+\mid\rF_{\tau_{\overline{x}}\wedge T})\ind_{\{\tau_{\overline{x}}<T,\, F_{\tau_{\overline{x}}}<\bar\rho\}}
\le \left(\frac{4x_L}{\overline{x}}\right)^2 = 16 \ee^{-2J}.
\]
Therefore, uniformly in $0\le s\le T$ and $0<x\le \eta_J$,
\begin{align*}
\bP_{s,x}\left(\widehat\sigma^\uparrow-s>Qx^2,\, \xi^L<L \right)
&\le \bP_{s,x}\left(\widehat\sigma^\uparrow-s>Qx^2,\ E_{J,x}\right) + \bP_{s,x}(\tau_{\underline{x}}\le \widehat\sigma^\uparrow) + \bP_{s,x}(A_+)\\
&\le \bP_{s,x}\left(\widehat\sigma^\uparrow-s>Qx^2,\ E_{J,x}\right) +32 \ee^{-2J},
\end{align*}
and thus, for every fixed \(J>1\),
\[
\lim_{Q\to\infty} \limsup_{x\downarrow 0} \sup_{0\le s\le T} \bP_{s,x}\left(\widehat\sigma^\uparrow-s>Qx^2,\ \xi^L<L \right) \le 32 \ee^{-2J}.
\]
Letting \(J\to\infty\) proves \eqref{eq:long-bad-cycle}.
\end{proof}

Recall the definition of $\omega^+_{b,T}(h)$ in \eqref{eq:positive-modulus}. For fixed $Q>0$, put
\begin{equation}\label{eq:trace-epsilon}
 \varepsilon_x(Q):= \frac{\omega^+_{b,T}(Q x^2)}x.
\end{equation}

\begin{lemma} \label{lem:frozen-trace}
Fix $Q,L>0$. There exists $x_0>0$ sufficiently small, such that for every $0<x<x_0$ and $0\le s< T$, there is a random variable $R$ independent of $\rF_s$ with \(\bP(R>r)=r^{-\bar{\nu}}\) for all $r\ge 1$, such that on the event $\{\widehat\sigma^\uparrow-s\le Qx^2,\, \xi^L<L\}$,
\begin{equation}\label{eq:trace-cycle-domination}
 \frac{F_{\sigma^\uparrow}}x \ge(1-\varepsilon_x(Q))R.
\end{equation}
\end{lemma}

\begin{proof}
Choose $x>0$ small enough that $\varepsilon_x(Q)<1$ and $x<\bar\rho \ee^{-L}$. 
Under $\bP_{s,x}$, freeze the level
\[
 \ell:=b(s)+\varepsilon_x(Q)x=W_s-(1-\varepsilon_x(Q))x.
\]
Recall $\fT^{(s)}$ in \eqref{eq:translation}. Put $a_0:=(1-\varepsilon_x(Q))x$ and define for $u\ge 0$,
\[
 V_u:=a_0+\fT^{(s)}_u(B)+\nu^*\sup_{0\le v\le u}\fT^{(s)}_v(B),
 \qquad u_0:=\inf\{u\ge0:V_u=0\}.
\]
The process $V$ is independent of $\rF_s$ by the strong Markov property of $B$.
Set $R:= M_{u_0}(V)/a_0$ so it is also independent of $\rF_s$. Applying \cite[Proposition 4(iii)]{PermanWerner} (with their parameters $(\alpha,\beta)=(\nu,0)$) to the process $V-a_0$, we deduce that $\bP(R>r)=r^{-\bar{\nu}}$ for all $r\ge 1$.

On the event $\{\widehat\sigma^\uparrow-s\le Qx^2,\, \xi^L<L\}$, we have $\sigma^\uparrow\le T\wedge(s+Qx^2)$ and the cycle cannot end at the outer cap $\bar\rho$. Therefore $W$ completes a full round-trip in $[s,\sigma^\uparrow]$. Before $W$ first hits $\ell$, $K$ remains constant so \eqref{eq:max-phase} yields $W_{s+u}-\ell=V_u$. On the other hand, the definition of $\varepsilon_x(Q)$ implies that $b(t)\le\ell$ for all $s\le t\le\sigma^\uparrow$, and thus $u_0+s\le \sigma^\uparrow$.  Consequently,
\[
F_{\sigma^\uparrow}=M_{\sigma^\uparrow}(W)-b(\sigma^\uparrow)
 \ge M_{s+u_0}(W)-\ell=M_{u_0}(V)=(1-\varepsilon_x(Q))xR.
\]
It gives \eqref{eq:trace-cycle-domination}.
\end{proof}

\begin{proposition}
\label{prop:cycle-kernel}
Fix $L>0$. Then uniformly over $0\le s\le T$ and
$0<x<\bar\rho$,
\begin{equation}\label{eq:lower-cycle-tail}
 \bP_{s,x}(\xi^L<-u)\le 16 \ee^{-2u}, \qquad u\ge0.
\end{equation}
Set $m_L:=(1-\ee^{-\bar\nu L})/\bar{\nu}.$ Then 
\begin{align}
\sup_{0\le s\le T,\,  0<x<\bar\rho} \bE_{s,x}\bigl[(\xi^L)^2\bigr] &<\infty,\label{eq:truncated-second-moment}\\
 \liminf_{x\downarrow0} \inf_{0\le s\le T} \bE_{s,x}\bigl[\xi^L\bigr]
 &\ge m_L.\label{eq:truncated-mean-uniform}
\end{align}
\end{proposition}

\begin{proof}
If $\xi^L<-u$, then $\sigma^\uparrow\le T$ and $F$ reaches $x\ee^{-u}$ before reaching $\bar\rho$ and before time $T$. Equation \eqref{eq:lower-cycle-tail} then follows by Proposition~\ref{prop:corner-power}. It follows that
\[
 \sup_{s\le T,\, 0<x<\bar\rho}\bE_{s,x}[((\xi^L)^-)^2]
 \le 32\int_0^\infty u \ee^{-2u}\dd u<\infty.
\]
Together with $(\xi^L)^+\le L$, this proves \eqref{eq:truncated-second-moment}.
Fix $Q>0$ and recall $\varepsilon_x(Q)$ in \eqref{eq:trace-epsilon}. By Lemma~\ref{lem:frozen-trace}, we have for all $x$ small enough, on the event $\{\sigma^\uparrow-s\le Qx^2,\, \xi^L<L\}$,
\[
 \xi^L\ge H^L:=\bigl(\log(1-\varepsilon_x(Q))+\log R\bigr)\wedge L.
\]
The same inequality is automatic on $\{\xi^L= L\}$.  Thus it can fail only on the event appearing in \eqref{eq:long-bad-cycle}. 
Therefore,
\[
 \xi^L\ge H^L - (|H^L|+|\xi^L|)\ind_{\{\sigma^\uparrow-s> Qx^2,\, \xi^L<L\}}.
\]
For every fixed $Q>0$, we can take $x>0$ sufficiently small such that $\varepsilon_x(Q)\le 1/2$, so that $-\log 2\le H^L \le L$.
From the uniform second-moment bound \eqref{eq:truncated-second-moment} and Cauchy--Schwarz, we deduce that
\[
 \bE_{s,x}\bigl[\xi^L\bigr] \ge \bE_{s,x} \bigl[H^L\bigr] - C_L \left[\sup_{0\le s\le T}\bP_{s,x}(\sigma^\uparrow-s> Qx^2,\, \xi^L<L)\right]^{\frac12}
\]
for $x>0$ sufficiently small and some constant $C_L>0$ independent of $s,x,Q$. We have $\varepsilon_x(Q)\to 0$ as $x\downarrow 0$ for fixed $Q>0$. Therefore, sending first $x\downarrow 0$ uniformly in $0\le s\le T$ and then $Q\to \infty$, Lemma~\ref{lem:long-bad-cycle} yields 
\[
\liminf_{x\downarrow 0} \inf_{0\le s\le T} \bE_{s,x}\bigl[\xi^L\bigr] \ge \bE[(\log R)\wedge L] = \int_0^L \ee^{-\bar\nu u}\dd u=\frac{1-\ee^{-\bar{\nu}L}}{\bar{\nu}}.
\]
This is \eqref{eq:truncated-mean-uniform}.
\end{proof}


\begin{proof}[Proof of Proposition~\ref{prop:round-trip-count}]
Fix \(A_1\in(\bar\nu,A)\). Note that \( m_L=(1-\ee^{-\bar\nu L})/\bar{\nu}\) increases to $1/\bar\nu$ as $L\to\infty$. Therefore we can take $L$ large enough and $\mu>0$ so that $1/A_1 < \mu< m_L$. 
Recall $Y:=M(W)-W$ in \eqref{eq:F-Y}. By Proposition~\ref{prop:cycle-kernel}, there exists \(\delta\in(0,\bar\rho)\) small enough such that for every stopping time \(\sigma\le T\), on the event \(\{0<F_\sigma\le\delta\}\),
\begin{align}\label{eq:cycle-drift}
\bE_{\sigma,F_{\sigma}} \bigl[\xi^L\bigr]\ge\mu,
\qquad
\bE_{\sigma,F_{\sigma}} \bigl[(\xi^L)^2\bigr]\le C_L.
\end{align}
By the strong Markov property of $B$, conditionally on $\rF_\sigma$ and on the event $\{Y_\sigma=0\}$, the process $(W_t:t\ge \sigma)$ has the law of $\bP_{\sigma, F_\sigma}$.

Start from width $F_s = x<\delta$ on the maximum face (i.e. $Y_s=0$), with $s<T$, and recall $\tau_\delta = \inf\{t>s:F_t=\delta\}.$ Let $\Sigma_0 = s$ and define $\Sigma_j$ recursively as follows: if $\Sigma_{j-1}<\tau_\delta\wedge T$, define $\Sigma_j = v_j\wedge T$ where $v_j$ is the terminal time when $W$ completes a cycle with outer cap $\bar\rho$ (note that in this case $W$ has completed $j-1$ cycles which are all round-trips); otherwise $\Sigma_j = \Sigma_{j-1}$.
Let $\chi$ be the first $j$ such that $\Sigma_{j} < \tau_\delta\wedge T \le \Sigma_{j+1}$, and put $\rG_j:= \rF_{\Sigma_j}$. Then the event $\{\chi\ge j\}=\{\Sigma_j<\tau_\delta\wedge T\}\in\rG_j$. Let $\xi^L_j$ be defined via \eqref{eq:tau-xi-truncated} with $(x,\sigma^\uparrow)$ replaced by $(F_{\Sigma_{j-1}}, v_j)$ there, and
\begin{align*}
Z_j:=\begin{cases}\xi_j^L, & \chi\ge j-1,\\
0, & \text{otherwise},
\end{cases}
\qquad m_j:=\bE_{s,x}(Z_j\mid\rG_{j-1}).
\end{align*}
Then we have $m_j\ge \mu\ind_{\{\chi\ge j-1\}}$ and $\bE_{s,x}(Z_j^2\mid\rG_{j-1})\le C_L \ind_{\{\chi\ge j-1\}}$  by the strong Markov property of $B$ and \eqref{eq:cycle-drift}. Hence the second moment of \(\cM_n:=\sum_{j=1}^n(Z_j-m_j)\) is bounded by $C_L n$. 
Put \(\ell_x:=\log(\delta/x).\) Then on the event $\{\chi>n\}$, $W$ completes $n$ round-trips before $T$ while $F$ does not reach $\delta$ in these round-trip intervals. Hence $x_j:= F_{\Sigma_j}<\delta$ and 
\[
\sum_{j=1}^n Z_j =\sum_{j=1}^n\left[\log \bigg(\frac{x_j}{x_{j-1}}\bigg)\wedge L\right] \le\log\frac{x_n}{x}
 <\ell_x.
\]
Consequently \(\cM_n\le\ell_x-n\mu\) on the event \(\{\chi> n\}.\) Chebyshev's inequality yields that, as $\eta\downarrow 0$,
\[
\sup_{0\le s<T,\, 0<x\le \eta}\bP_{s,x}(\chi>\lceil A_1\ell_x\rceil) \le \sup_{0\le s<T,\, 0<x\le \eta}\bP_{s,x}(\cM_{\lceil A_1\ell_x\rceil}\le -(A_1 \mu -1) \ell_x) \le \frac{C_L (A_1\ell_\eta + 1)}{(A_1\mu -1)^2 \ell_\eta^2} \to 0.
\]
On the other hand, on the event \(\{\tau_\delta\le T\}\), we have
\(F_{\tau_\delta}=\delta\) and \(Y_{\tau_\delta}\in[0,\delta]\).
Denote by $P_{\delta}$ the number of round-trips of $W$ in $[\tau_\delta,T]$ and set $P_\delta:=0$ if $\tau_\delta\ge T$. We observe that $\fN_T\le (\chi+1) + P_\delta$. Let $q_x:=A\log(1/x)-\lceil A_1\ell_x\rceil-1$ which tends to $\infty$ as $x\downarrow 0$. It follows that
\begin{align*}
&\sup_{0\le s<T,\, 0<x\le \eta}\bP_{s,x}(\fN_T>A\log(1/x))\\
\le\, & \sup_{0\le s<T,\, 0<x\le \eta}\bP_{s,x}(\chi > \lceil A_1\ell_x\rceil) + \sup_{0\le s<T,\, 0<x\le \eta}\bE_{s,x}\bigl[\bP_{s,x}(P_{\delta}>q_x\mid \rF_{\tau_\delta\wedge T}) \bigr]\\
\to &\ 0, \qquad \text{as}~ \eta\downarrow 0,
\end{align*}
where the limit of the second part follows by Proposition~\ref{prop:positive-width-tightness}, with restart law $\bP^\delta_{\tau_\delta,\delta-Y_{\tau_\delta}}$ on the event $\{\tau_\delta<T\}$. This completes the proof.
\end{proof}

\section{Proof of Theorem~\ref{thm:main}}\label{sec:proof}

We prove Theorem~\ref{thm:main} in this section. The case $x>0$ in \eqref{eq:regulator} has been solved in Proposition~\ref{prop:corner-power} and Proposition~\ref{prop:positive-gap}, so it suffices to prove the theorem in the case $x=0$.
The case $\nu<1/2$ can be solved for example by applying Picard's iteration and the fixed point theorem, as introduced by Yor \cite[Section 8.5]{Yor1992}. We can also view it as a simple consequence of Williams \cite[Theorem 2.1]{Williams1995}. Consider the Skorokhod problem in an orthant as shown in \eqref{eq:2D-Skorokhod}. For $Q=I-R$, the spectral radius of $|Q|$ is $\sqrt {|\nu^*|}<1$, so \eqref{eq:2D-Skorokhod} admits a unique $\rF$-adapted solution $(X,Y,K,V)$. 
Set $W=X+b$. Adding the two coordinate equations gives \(Y=x-W+V/\bar\nu.\) Since $V/\bar\nu$ is a continuous increasing function starting from 0 and increases only when $Y=0$, Skorokhod's lemma (in the one-dimensional case) yields
\[
\frac{1}{\bar\nu}V_t=\sup_{0\le s\le t}(W_s-x)^+ =M_t(W)-x,\qquad t\ge 0,
\]
where the last equality uses $W_0=x$. Substituting this equality into \eqref{eq:2D-Skorokhod} gives \eqref{eq:regulator}, and the identification $K=\frac12L^0(W-b)$ follows by Proposition~\ref{prop:positive-gap}.

\bigskip

Note that in the critical case $\nu=1/2$, we have $\nu^*=1$ in the matrix $R$ in \eqref{eq:2D-Skorokhod} hence the same conclusion holds by \cite[Theorem 3.2]{Bass2024}.
We give a unified argument for $\nu\ge 1/2$; the proof is an application of  Propositions~\ref{prop:comparison} and \ref{prop:round-trip-count}, following the strategy of \cite[Theorem 1]{ChaumontDoney}. 
We first choose a sequence of Brownian stopping times $(u_n:n\ge 1)$ decreasing to 0 and modify the driving function $B$ on $[0,u_n]$, thereby obtaining a sequence of solutions $(W^{(n)}:n\ge 1)$ with driving functions $B^{(n)}$ that agree with $B$ after $u_n$. We then apply Propositions~\ref{prop:comparison} and \ref{prop:round-trip-count} to show that $(W^{(n)}:n\ge 1)$ is a Cauchy sequence.
The locally uniform limit solves \eqref{eq:regulator}, and the decreasing stopping times recover adaptedness. Pathwise uniqueness follows from the same comparison estimate.

\subsection{Existence and adaptedness}
Recall $\bar\nu = 1-\nu\in(0,1/2]$ and $\nu^* = \nu/\bar\nu\in[1,\infty)$. We can take $A>A_0>\bar\nu$ such that $2A \log (\nu^*) <1$; indeed $2\bar\nu\log (\nu^*) = 2\log (\nu^*)/(1+\nu^*)<1$ for all $\nu^* \in[1,\infty)$. Put  \(\kappa:=1-2A\log \nu^*>0\) and set
\begin{equation}\label{eq:epsilon-n}
 \varepsilon_n:=2^{-n},  \qquad \omega_b^0(h):=\sup_{0\le t\le h}|b(t)|.
\end{equation}

\noindent Recall the definition of $\omega^+_{b,T}(h)$ in \eqref{eq:positive-modulus}. Put $q_0=r_0=1$.  Choose a positive and decreasing sequence $(q_n:n\ge 1)$  $q_n>0$ such that $q_n\downarrow 0$ as $n\to \infty$ and, with $r_n:=\varepsilon_n\sqrt{q_n}$,
\begin{gather}
 \quad \quad n^2r_n^\kappa\le \varepsilon_n,\qquad   \omega^+_{b,1}(q_n)\le\varepsilon_n^2\sqrt{q_n} =\varepsilon_n r_n, \label{eq:qn-boundary}\\
    \sup_{\substack{0\le s\le 1\\ x\in[r_n/(2\bar\nu),\,r_n/\bar{\nu} + \omega_b^0(q_n)]}}  \bP_{s,x}\left(\fN_{n}>A\log(1/r_n)\right)\le \varepsilon_n,\label{eq:rn-cycle}
\end{gather}
where $\bP_{s,x}$ and $\fN_T$ are defined at the beginning of Section~\ref{sec:cycles}. 
This recursive choice is possible: \eqref{eq:PB} gives the second condition in \eqref{eq:qn-boundary}. Moreover, for $r_n$ small and $x\in[r_n/(2\bar\nu),r_n/\bar{\nu} + \omega_b^0(q_n)]$, one has $x\ge r_n$ because $2\bar\nu\le1$, and hence  \(A\log(1/r_n)\ge A_0\log(1/x).\) Continuity of $b$ at zero gives $\omega_b^0(q_n)\to 0$, and therefore \eqref{eq:rn-cycle} follows by Proposition~\ref{prop:round-trip-count} with $T=n$ whenever we take $q_n$ sufficiently small.

Let \(u_n:=\inf\{t\ge0:B_t=r_n\}.\) We have \(u_m\le u_n\) for \(m>n\) and
\(u_n\downarrow 0\) almost surely as $n\to\infty$.  Define the record event
\begin{equation}\label{eq:record-good}
 G_n:=
 \left\{u_n\le q_n,\quad
 \inf_{0\le s\le u_n}B_s>-n^2r_n\right\}.
\end{equation}
On the event $G_n$, we have
\begin{equation}\label{eq:b-at-record}
 -\omega_b^0(q_n)\le b(u_n)\le \omega^+_{b,1}(q_n)\le\varepsilon_n r_n.
\end{equation}
We claim that $G_n$ occurs for all sufficiently large $n$ a.s. Indeed, the reflection principle gives
\[
 \bP(u_n>q_n) =\bP\left(\sup_{s\le q_n}B_s<r_n\right) \le C\frac{r_n}{\sqrt{q_n}} =C\varepsilon_n,
\]
while the gambler's-ruin formula applied between \(r_n\) and \(-n^2r_n\)
gives
\[
 \bP\left(\inf_{s\le u_n}B_s\le-n^2r_n\right) =\frac1{1+n^2}.
\]
Therefore $\sum_n \bP(G_n^c)<\infty$, and the claim follows by the Borel--Cantelli Lemma.

We next put \(b^\uparrow(t):=\sup_{0\le s\le t } b(s),\) which is continuous and increasing.  On the event $G_n$, we have $b^\uparrow(u_n)\le\varepsilon_n r_n  < r_n/\bar{\nu}$.  Define, for
$0\le t\le u_n$,
\begin{equation*}
\begin{aligned}
 W_t^{(n)} &= M_t(W^{(n)}) := b^\uparrow(t) + \left(\frac{r_n}{\bar{\nu}}-b^\uparrow(u_n)\right)\frac{t}{u_n}, \\
 B_t^{(n)}&:=\bar\nu\, W_t^{(n)}, \qquad K_t^{(n)}:=0.
\end{aligned}
\end{equation*}
The path $W^{(n)}$ is increasing, lies above $b$, and satisfies
$W^{(n)}=B^{(n)}+\nu M(W^{(n)})$.  Extend the modified driver by
$B_t^{(n)}=B_t$ for $t\ge u_n$; it is continuous because both values at
$u_n$ equal $r_n$.  At the record time,
\begin{equation}\label{eq:fn}
 W_{u_n}^{(n)}=M_{u_n}(W^{(n)})=\frac{r_n}{\bar{\nu}},
 \qquad
 f_n:=F_{u_n}^{(n)}
 =\frac{r_n}{\bar{\nu}}-b(u_n)
 \in\left[\frac{r_n}{2\bar\nu},\frac{r_n}{\bar{\nu}} + \omega_b^0(q_n)\right].
\end{equation}
The lower bound follows from \eqref{eq:b-at-record} and
$\varepsilon_n\le1/(2\bar\nu)$; the upper bound follows from the definition of $\omega_b^0$.
Continue $W^{(n)}$ after \(u_n\) so that $W^{(n)}_{\cdot+u_n}-b(u_n)$ is the unique solution of \eqref{eq:regulator} with driving function \(B_{u_n+\cdot}-B_{u_n}\) and boundary $b(u_n+\cdot)-b(u_n)$ and starting from $f_n$.   
To define the approximations on the whole sample space, set \(W^{(n)}=K^{(n)}=B^{(n)}=0\) on \(G_n^c\).  
Let \(\fN^{(n)}\) denote the number of completed round-trips of \(W^{(n)}\) in $[u_n,n]$ on the event \(G_n\) and put \(\fN^{(n)}=0\) on the complementary event, and set \(C_n = G_n^c\cup \{\fN^{(n)}\le A\log(1/r_n)\}.\) The strong Markov property of $B$ at \(u_n\), the width range in \eqref{eq:fn}, and \eqref{eq:rn-cycle} give  \(\bP(C_n^c)\le \varepsilon_n.\) Hence the Borel--Cantelli Lemma yields that $C_n$ holds for all sufficiently large $n$ a.s. We can work henceforth on the fixed probability-one event \(\Omega_*:=\liminf_{n\to\infty} (C_n\cap G_n)\). 

For a sample point in \(\Omega_*\), choose
\(N=N(\omega)\) so that \(G_j\cap C_j\) holds for every \(j\ge N\), and take \(m>n\ge N\). 
We have \(B_s^{(m)}\le r_n\) and \(K_s^{(m)}\le K_{u_n}^{(m)}\) for \(0\le s\le u_n\), and both terms attain their upper bounds at \(u_n\). Corollary~\ref{cor:max} with zero initial position and driving function $B^{(m)}$ for \(W^{(m)}\) therefore gives
\[
 W_{u_n}^{(m)}=M_{u_n}(W^{(m)}), \qquad \bar\nu M_{u_n}(W^{(m)})=r_n+K_{u_n}^{(m)}.
\]
Equation \eqref{eq:K-rough} gives, uniformly over \(m>n\ge N(\omega)\),
\begin{align} \label{eq:record-regulator-bound}
 K_{u_n}^{(m)} \le \left[\sup_{s\le u_n}b(s) - \inf_{s\le u_n}B_s^{(m)}\right]^+ \le \varepsilon_n r_n+n^2r_n \le 2 n^2r_n.
\end{align}
Indeed, \(\sup_{s\le u_n}b(s)\le \omega^+_{b,1}(q_n)\le
\varepsilon_nr_n\) and \((B^{(m)}_s:0\le s\le u_m)\) is nonnegative, while \(B^{(m)}=B\) after \(u_m\) and \(\inf_{u_m\le s\le u_n}B_s>-n^2r_n\) on \(G_n\).
Restart $W^{(n)}$ and $W^{(m)}$ at \(u_n\) and use the common driver \(B_{u_n+\cdot}-B_{u_n}\) and boundary \(b(u_n+\cdot)-b(u_n)\). Their positions at time $u_n$ are respectively $r_n/\bar\nu$ and
\[
 \frac{r_n}{\bar\nu} + d_{m,n},\qquad
 d_{m,n}:=\frac{K_{u_n}^{(m)}}{\bar{\nu}}\le \frac2{\bar\nu} n^2 r_n.
\]
Proposition~\ref{prop:comparison} therefore yields that, for $m>n\ge N(\omega)\vee T$ large enough,
\begin{equation}\label{eq:cauchy-after-record}
 \sup_{u_n\le t\le T}|W_t^{(m)}-W_t^{(n)}|\le \frac2{\bar\nu}n^2r_n\,(\nu^*)^{2\fN^{(n)}+1}.
\end{equation}
 On the event \(C_n\), the right side of \eqref{eq:cauchy-after-record} is bounded by \( (2\nu^*/\bar\nu) n^2 r_n^{1-2A\log \nu^*} = (2\nu^*/\bar\nu) n^2 r_n^\kappa\) which tends to 0 as $n\to\infty$ from \eqref{eq:qn-boundary}.
In the time interval \([0,u_n]\), we have \(0\le W^{(n)}\le r_n/\bar{\nu}\), while  \(-\omega_b^0(q_n)\le b\le W^{(m)}\le M_{u_n}(W^{(m)})\le (1+2n^2) r_n/\bar\nu.\) Since $0<\kappa\le1$, \eqref{eq:qn-boundary} also implies
$n^2r_n\to 0$.  Since $\omega_b^0(q_n)\to0$, the same uniform convergence
holds on the initial interval $[0,u_n]$. We have
proved on \(\Omega_*\), and hence almost surely, that \((W^{(n)}:n\ge 1)\) is
a Cauchy sequence of continuous functions under the supremum norm on any compact interval.
The running-maximum map is one-Lipschitz in the supremum norm on compact intervals, so \(M(W^{(n)})\) is Cauchy as well, and hence so is the sequence
\[
 K^{(n)}=W^{(n)}-B^{(n)}-\nu M(W^{(n)}),\qquad n\ge 1.
\]
Since \(u_n\downarrow 0\), we obtain that $B^{(n)}$ converges to $B$ uniformly on compact intervals as $n\to \infty$. Denote by \((W,M,K)\) the limit of \((W^{(n)},M(W^{(n)}),K^{(n)})\).  Then \(M=M(W)\) and
\begin{equation}\label{eq:limit-regulator}
 W_t=B_t+\nu M_t(W)+K_t,\qquad W_t\ge b(t).
\end{equation}
The uniform limit \(K\) is continuous and nondecreasing, with \(K_0=0\).
Let \(I\subseteq\{t\ge 0:W_t>b(t)\}\) be a compact interval. Then \(W^{(n)}>b\) on \(I\) for all large \(n\), so every \(K^{(n)}\), and hence \(K\), remains constant on \(I\).  Since \(\{W>b\}\) is a countable union of open intervals,
\[
 \int_0^\infty \ind_{\{W_t>b(t)\}}\dd K_t=0.
\]
Therefore \((W,K)\) is a global solution of \eqref{eq:regulator} for $x=0$. Each approximation $W^{(n)}$ is $(\rF_{u_n\vee t}:t\ge 0)$-adapted since $G_n\in \rF_{u_n}$. Since $u_n\downarrow 0$, we have $\cap_{n\ge 1}\rF_{u_n\vee t} = \rF_t$ for every $t\ge 0$. It follows that $(W,K)$ is $\rF$-adapted. We moreover have \(K=\frac12 L^0(W-b)\) by Proposition~\ref{prop:positive-gap}. Thus the constructed process $W$ solves \eqref{eq:main}.


\subsection{Pathwise uniqueness}

Let \(Z\) be any other solution of \eqref{eq:main} driven by the same Brownian motion $B$, and set \(K_t^Z:=\frac12L_t^0(Z-b)\) so that \((Z,K^Z)\) is a solution to \eqref{eq:regulator}. 
At time \(u_n\), Corollary~\ref{cor:max} gives
\begin{equation}\label{eq:other-record}
 Z_{u_n}=M_{u_n}(Z),\qquad
 \bar\nu M_{u_n}(Z)=r_n+K_{u_n}^Z.
\end{equation}
On the event \(G_n\), using
\(\sup_{s\le u_n}b(s)\le\omega^+_{b,1}(q_n)\le\varepsilon_nr_n\), we derive that
\begin{equation}\label{eq:other-K-bound}
 K_{u_n}^Z \le\left[\sup_{s\le u_n}b(s)-\inf_{s\le u_n}B_s\right]^+ \le 2n^2r_n,
\end{equation}
while \(Z_{u_n} - b(u_n) = (r_n+K_{u_n}^Z)/\bar\nu - b(u_n)\ge r_n/\bar\nu -\varepsilon_n r_n>0.\) By Proposition~\ref{prop:corner-power}, $M(Z)-b>0$ after \(u_n\).
Compare \(Z\) after \(u_n\) with \(W^{(n)}\). They have the same boundary $b(\cdot+u_n)-b(u_n)$ and are both driven by $B_{\cdot+u_n}-B_{u_n}$, and the difference of their starting positions at time $u_n$ is
\(K_{u_n}^Z/\bar{\nu}\). Therefore Proposition~\ref{prop:comparison} gives
\[
 \sup_{u_n\le t\le T}|Z_t-W_t^{(n)}|\le \frac{2\nu^*}{\bar\nu} n^2 r_n^\kappa \to 0
 \quad\text{as}~ n\to\infty.
\]
On \([0,u_n]\), we have \(-\omega_b^0(q_n)\le b(t)\le Z_t\le M_{u_n}(Z)\le (1+2n^2)r_n/\bar\nu \) from
\eqref{eq:other-record}--\eqref{eq:other-K-bound}, while \(0\le W^{(n)}\le r_n/\bar{\nu}\).
Hence the convergence is indeed uniform on \([0,T]\). Since \(W^{(n)}\to W\) uniformly on compact intervals, \(Z=W\) on \([0,T]\) for all $T\ge 0$, which proves the pathwise uniqueness of \eqref{eq:main} and completes the proof of Theorem~\ref{thm:main}.

\section{Proof of Theorem~\ref{thm:nonexistence}}\label{sec:example}

In this section, we construct the reflecting boundary $b$ satisfying Theorem~\ref{thm:nonexistence}. We first prove the following necessary no-charge property, complementary to Proposition~\ref{prop:positive-gap}.
\begin{lemma}\label{lem:necessary-nocharge}
Let $b$ be a deterministic continuous increasing function. If an adapted continuous process $W$ solves \eqref{eq:main} up to time $T>0$ for some $\nu<1$, then almost surely,
\begin{equation}\label{eq:nocharge}
 \int_0^T\ind_{\{W_s=b(s)\}}\dd b(s)=0.
\end{equation}
\end{lemma}
\begin{proof}
Put $X=W-b\ge0$, $M=M(W)$ and $K=\tfrac12 L^0(X)$. The process $K$ is continuous and increasing, and its measure is a.s. carried by $\{X=0\}$. Since the quadratic variation $\langle X\rangle_t = t$, the occupation times formula gives \(\int_0^t\ind_{\{X_s=0\}}\dd s=0\) and thus $\int_0^t\ind_{\{X_s=0\}}\dd B_s = 0$ for all $0\le t\le T$. Define 
\[
 A_t=\int_0^t\ind_{\{W_s=b(s)\}}\dd b(s),\qquad
 J_t=\int_0^t\ind_{\{W_s=b(s)\}}\dd M_s,\qquad 0\le t\le T.
\]
Tanaka's formula for $X=X^+$ gives, for every $0\le t\le T$,
\[
K_t = \frac12L^{0}_t(X) =\int_0^t\ind_{\{X_s=0\}}\dd X_s =\nu J_t+K_t-A_t,
\]
and hence $A_t=\nu J_t$.
On the other hand, the continuous function $D=M-b$ is nonnegative and locally of finite variation, which implies  $\ind_{\{D=0\}}\dd D=0$ as a signed measure and hence \(\ind_{\{M=b\}}\dd M=\ind_{\{M=b\}}\dd b.\) 
Because $M\ge W\ge b$ and $\dd M$ is carried by $\{W=M\}$, we deduce that
\[
 J_t=\int_0^t\ind_{\{M_s=W_s=b(s)\}}\dd M_s
     =\int_0^t\ind_{\{M_s=b(s)\}}\dd b(s)\le A_t,\qquad 0\le t\le T.
\]
Since $\nu<1$ and $A_t=\nu J_t$, we obtain that $A_t=J_t=0$ for all $0\le t\le T$, which proves \eqref{eq:nocharge}.
\end{proof}

Fix $\alpha\in(0,1/2)$ and set $r:=2^{-1/\alpha}\in(0,1/4)$. Define
\[
    S_0(x):=rx,\qquad
    S_1(x):=1-r+rx,\qquad x\in[0,1].
\]
For a finite word $\mathbf{i}=(i_1,\ldots,i_n)\in\{0,1\}^n$, let
\[
    S_\mathbf{i} :=S_{i_1}\circ\cdots\circ S_{i_n},
    \qquad
    I_\mathbf{i}:=S_\mathbf{i}([0,1]).
\]
We have $I_\mathbf{i}=[a_\mathbf{i},a_\mathbf{i}+r^n]$ where $a_\mathbf{i}:=(1-r)\sum_{k=1}^n i_k r^{k-1}$. Set $C_r^{(0)}:=[0,1]$ and then
\[
    C_r^{(n)}
    :=\bigcup_{\mathbf{i}\in\{0,1\}^n}I_\mathbf{i},
    \qquad
    \mathcal C_r:=\bigcap_{n=0}^\infty C_r^{(n)}.
\]
Equivalently,
\[
    \mathcal C_r =  \left\{ (1-r)\sum_{k=1}^\infty a_k r^{k-1}: a_k\in\{0,1\} ,\, \forall\, k\ge 1\right\}.
\]
Let $(\varepsilon_j:j\ge 1)$ be a sequence of i.i.d. Bernoulli random variables with parameter $1/2$, and $\mu$ be the law of $(1-r)\sum_{j=1}^{\infty}\varepsilon_j r^{j-1}$. Then $\mu$ is supported on $\mathcal C_r$, and $\mu(I_\mathbf{i})=2^{-n}$ for all $\mathbf{i}\in\{0,1\}^n$. We define
\begin{equation}\label{eq:boundary}
 b(t)=\begin{cases}\mu([0,t]), & 0\le t\le1,\\
 1, &t> 1.
 \end{cases}
\end{equation}
The measure $\mu$ has no atoms, since for every $x\in\mathcal C_r$ and $n\ge 1$ there exists a unique $\mathbf{i}\in\{0,1\}^n$ such that $x\in I_\mathbf{i}$, which implies $\mu(\{x\})\le 2^{-n}$ for every $n\ge 1$. It is moreover singular with respect to Lebesgue measure since $|C_r^{(n)}|=2^n r^n$ which tends to 0 as $n\to\infty$. 
We now prove that $b$ is $\alpha$-H\"older continuous. Let $I$ be an interval of length $h\in(0,1]$ and choose $k\ge 0$ such that $r^{k+1}<h\le r^k$. For $k\ge 1$ and two words $\mathbf{i}\ne \mathbf{j}\in \{0,1\}^k$, the distance between $I_\mathbf{i}$ and $I_\mathbf{j}$ is at least $(1-2r)r^{k-1}>r^k$ because $r<1/4$. Therefore there exists at most one $\mathbf{i}\in\{0,1\}^k$ such that $I\cap I_\mathbf{i}\ne\emptyset$, and thus 
\[
\mu(I)\le 2^{-k} = r^{k\alpha} \le (r^{-1} h)^{\alpha} = 2 h^{\alpha}.
\]
The bound also holds trivially for $h>1$, and the case $k=0$ is immediate. It proves the $\alpha$-H\"older continuity of $b$. Since $\dd b = \mu$,  $b$ satisfies condition~(i) in Theorem~\ref{thm:nonexistence}.
The construction also gives
\begin{equation}\label{eq:self-similarity}
 \mathcal C_r\cap[0,r^n]=r^n\mathcal C_r,
 \qquad b(r^n)=2^{-n}=r^{n\alpha}.
\end{equation}
In particular, $(b(r^n)-b(0))/\sqrt{r^n} = r^{n(\alpha-1/2)}\to \infty$ as $n\to \infty$. Thus $b$ fails \eqref{eq:PB} already at the initial time. 

We prove the following result as a key ingredient of the proof of Theorem~\ref{thm:nonexistence}, showing that the contact set $\{t:w_t=b(t)\}$ has positive $\dd b$-measure.

\begin{proposition}\label{prop:contacts}
Fix $\nu<1$, $T_0\in(0,1]$ and $\gamma\in(\alpha,1/2)$. Let $e:[0,\infty)\to\bR$ satisfy $e(0)=0$, and suppose that there exists some $H>0$ such that
\begin{align}\label{eq:holder-e}
 |e(t)-e(s)|\le H|t-s|^\gamma,\qquad\forall\, 0\le s,t\le T_0.
\end{align}
Suppose continuous paths $w,k^w$ satisfy the equation
\begin{equation}\label{eq:det-regulator}
 w_t=e(t)+\nu M_t(w)+k^w_t,\qquad w_t\ge b(t),\qquad w_0=0,
\end{equation}
where $k^w$ is a continuous increasing function starting from 0 such that $\dd k^w$ is supported on $\{w=b\}$. Then there exists a constant $C>0$ depending only on $\gamma,r$ and $\nu$, such that for every $n\ge 1$ with $r^n\le T_0$,
\begin{equation}\label{eq:contact-bound}
 \int_0^{r^n}\ind_{\{w_t=b(t)\}}\dd b(t) \ge r^{n\alpha} - C H r^{n\gamma}.
\end{equation}
In particular, \(\int_0^T\ind_{\{w_t=b(t)\}}\dd b(t)>0\) for every $T>0$.
\end{proposition}

\begin{proof}[Proof of Theorem~\ref{thm:nonexistence} assuming Proposition~\ref{prop:contacts}]
Fix  $\nu<1$ and $T > 0$, set $T_0 = T\wedge 1$, and choose $\gamma\in(\alpha,1/2)$. Almost surely, the Brownian motion $B$ satisfies \eqref{eq:holder-e} on $[0,T_0]$ with a finite random constant $H$. If a solution $W$ existed, Proposition~\ref{prop:contacts} would give positive $\dd b$-measure to its contact set $\{W=b\}$ on $[0,r^n]$ for all sufficiently large $n$, contradicting Lemma~\ref{lem:necessary-nocharge}.
\end{proof}

The remainder of this section is devoted to the proof of Proposition~\ref{prop:contacts}. We first establish a packing estimate. Note that $2r^{\gamma}<2r^\alpha=1$ for $\gamma>\alpha$.

\begin{lemma}\label{lem:packing}
Let $n\ge 0$, $\gamma\in(\alpha,1/2)$ and let $(O_j:j\ge 1)$ be intervals contained in $[0,r^n]$ with pairwise disjoint interiors and positive $\mu$-measures. Then
\begin{equation}\label{eq:packing}
 \sum_{j\ge 1} |O_j|^\gamma \le \frac4{1-2r^\gamma}r^{n\gamma}.
\end{equation}
\end{lemma}
\begin{proof}
For $k\ge 0$, consider the collection of intervals $(O_j)$ with
$r^{n+k+1}<|O_j|\le r^{n+k}$. By \eqref{eq:self-similarity},
$\mathcal C_r\cap[0,r^n]$ is covered by $2^{k+1}$ intervals of length $r^{n+k+1}$. Each $O_j$ in this group meets at least one $I_{\mathbf{i}}$ with $\mathbf{i} \in\{0,1\}^{n+k+1}$ and $i_1=\cdots=i_n=0$. On the other hand, since $|O_j|>r^{n+k+1}$, any such interval $I_{\mathbf{i}}$ can meet at most two of these $O_j$'s; otherwise, among three ordered intervals in this group meeting $I_{\mathbf{i}}$, the middle one would have length at most $r^{n+k+1}$, a contradiction.
Thus there are at most $2^{k+2}$ intervals in this group. Summing their $\gamma$-th powers gives
\[
 \sum_j|O_j|^\gamma
 \le\sum_{k=0}^{\infty}2^{k+2}r^{(n+k)\gamma}
 =\frac4{1-2r^\gamma}r^{n\gamma},
\]
which is \eqref{eq:packing}.
\end{proof}

\begin{proof}[Proof of Proposition~\ref{prop:contacts}]
Fix $n\ge 1$ with $r^n\le T_0$. Consider a connected component $I$ of $\{t\in[0,r^n]:w_t>b(t)\}$ with endpoints $u<v$ (we may have $v=r^n$ so $w_v\ge b(v)$ in general). Then we have $w_u=b(u)$ and $k^w_t=k^w_u$ for all $t\in[u,v]$. Put $h=M_u(w)-w_u\ge0$ and recall $\bar\nu=1-\nu>0$. For $u\le t\le v$, we have $w_t-w_u=e(t)-e(u)+\nu(M_t(w)-M_u(w))$ from \eqref{eq:det-regulator} and therefore 
\[
M_t(w)-w_t=h-(e(t)-e(u))+\bar\nu(M_t(w)-M_u(w)).
\]
Since $\dd M(w)$ is carried by $\{w=M(w)\}$, Skorokhod's lemma gives
\begin{equation}\label{eq:one-face}
M_t(w)-M_u(w)=\frac1{\bar\nu} \left[\sup_{u\le s\le t}(e(s)-e(u))-h\right]^+.
\end{equation}
If $\nu\ge 0$, \eqref{eq:holder-e} and \eqref{eq:one-face} imply $w_v-w_u\le\bar\nu^{-1}H(v-u)^\gamma$; if $\nu<0$, \eqref{eq:det-regulator} gives $w_v-w_u\le H(v-u)^\gamma$.
Consequently, 
\begin{equation}\label{eq:excursion-mass}
 \mu(I)=b(v)-b(u) \le w_v-w_u \le\max\{1,\bar\nu^{-1}\}H|I|^\gamma.
\end{equation}
Let $E=\{t\in[0,r^n]:w_t=b(t)\}$. The complement of $E$ is a countable union of the components just considered. Applying Lemma~\ref{lem:packing} and \eqref{eq:excursion-mass} yields
\[
 \mu([0,r^n]\backslash E)
 \le \frac4{1-2r^\gamma}\max\{1,\bar\nu^{-1}\}H r^{n\gamma}.
\]
Subtracting this from $\mu([0,r^n])=r^{n\alpha}$ gives \eqref{eq:contact-bound}. Since $\gamma>\alpha$, the right-hand side of \eqref{eq:contact-bound} is positive for all sufficiently large $n$.
\end{proof}


\bibliographystyle{abbrv}
\bibliography{references_new}

\end{document}